\documentclass{article} 

\usepackage{amssymb}
\usepackage{amsthm}
\usepackage{amsfonts}
\usepackage{amsmath}
\usepackage{color}
\usepackage{url}
\usepackage{graphicx}
\usepackage{tikz}
\usepackage{lipsum}
\usepackage{dsfont}
\usepackage{booktabs}

\usepackage{authblk}

\renewenvironment{proof}{\noindent {\bf Proof:}}{\hfill $\Box$}

\newtheorem{theorem}{Theorem}
\newtheorem{lemma}{Lemma}
\newtheorem{proposition}{Proposition}
\newtheorem{corollary}{Corollary}
\newtheorem{Algorithm}{Algorithm}

\newtheorem{definition}{Definition}
\newtheorem{assumption}{Assumption}
\newtheorem{remark}{Remark}

\newtheorem{example}{Example}
\usepackage{booktabs}
\usepackage{color}
\usepackage{bm}
\usepackage{sidecap}

\usepackage{algorithm}
\usepackage[noend]{algpseudocode}
\makeatletter
\def\BState{\State\hskip-\ALG@thistlm}
\makeatother
\usepackage{algorithmicx}

\usepackage{amsmath}
\usepackage{amssymb}
\usepackage{amsfonts}
\usepackage{hyperref}
\usepackage{cancel}

\definecolor{dkgreen}{rgb}{0,0.4,0}
\definecolor{dkred}{rgb}{0.8,0,0}

\usepackage[normalem]{ulem}

\usepackage{tikz}

\newcommand{\dist}{\mathrm{dist}}
\newcommand{\A}{\mathcal{A}}
\newcommand{\I}{\mathrm{I}}
\newcommand{\T}{\mathrm{T}}

\newcommand{\R}{\mathbb{R}}
\newcommand{\Cc}{\mathcal{C}}
\newcommand{\N}{\mathbb{N}}

\title{Sparse decompositions of nonlinear dynamical systems and applications to moment-sum-of-squares relaxations\thanks{This work has been supported by European Union's Horizon 2020 research and innovation programme under the Marie Sk\l{}odowska-Curie Actions, grant agreement 813211 (POEMA).}}

\begin{document}

\author[1]{Corbinian Schlosser}
\author[2,3]{Milan Korda}
\affil[1]{INRIA - Ecole Normale Supérieure - PSL Research university, Paris, France. {\tt corbinian.schlosser@inria.fr}}
\affil[2]{LAAS-CNRS, 7 avenue du colonel Roche, F-31400 Toulouse; France. {\tt korda@laas.fr}}
\affil[3]{Faculty of Electrical Engineering, Czech Technical University in Prague, Technick\'a 2, CZ-16626 Prague, Czech Republic.}

\maketitle


\begin{abstract}
In this paper, we propose a general sparse decomposition of dynamical systems provided that the vector field and constraint set possess certain sparse structures, which we call subsystems. This notion is based on causal dependence in the dynamics between the different states. This results in sparse descriptions for fundamental problems from nonlinear dynamical systems:  region of attraction, maximum positively invariant set, and global attractor. The decompositions can be paired with any method for computing (outer) approximations of these sets in order to reduce the computation to lower dimensional systems. We illustrate this by the methods from previous work based on infinite-dimensional linear programming. This exhibits one example where the curse of dimensionality is present and hence dimension reduction is crucial. In this context, for polynomial dynamics, we show that these problems admit a sparse sum-of-squares (SOS) approximation with guaranteed convergence such that the number of variables in the largest SOS multiplier is given by the dimension of the largest subsystem appearing in the decomposition. The dimension of such subsystems depends on the sparse structure of the vector field and the constraint set; if the dimension of the largest subsystem is small compared to the ambient dimension, this allows for a significant reduction in the computation time of the SOS approximations.
\end{abstract}

\textbf{Keywords:} nonlinear control, large scale optimization, 
 sparsity, decomposition, network control, sum-of-squares


\section{Introduction}
\label{sec:introduction}
Many tasks involving dynamical systems are of computationally complex nature and often not tractable in high dimensions. Among these are the computations of the region of attraction (ROA), maximum positively invariant (MPI) set, and global attractors (GA), all of which are the focus of this work. These sets are ubiquitous in the study of dynamical systems and have numerous applications. For example, the ROA is the natural object to certify which initial values will be steered to a desired configuration after a finite time $T$ while the solution trajectory satisfies the state constraints at all times. The question of which initial values will stay in the constraint set for all positive times is answered by the MPI set. The GA describes which configurations will be reached uniformly by the solutions of the dynamical system asymptotically.  Since these objects are complex in nature, their computation is a challenging task. Computational methods for the ROA have been pioneered by Zubov~\cite{Zubov} in the 1960s and have a long history, summarized in~\cite{Chesi}. A survey on the (controlled) MPI set and computational aspects can be found in~\cite{Blanchini}. Computations of the GA are typically approached via Lyapunov functions~\cite{LyapunovFunctions}, via finite-time truncation or set-oriented methods~\cite{Dellnitz}.


Given the curse of dimensionality problem present in the computation of these sets, it is important to exploit structure in order to reduce the complexity. There are several concepts used for reducing the complexity, such as leveraging symmetries (see, e.g., \cite{FantuzziGoluskin}) or the knowledge of Lyapunov or Hamilton functions (see, e.g.,~\cite{Valmorbida}). In this text, we investigate a specific type of sparsity found in dynamical systems. We consider a framework for which we proposed sparse computations of invariant measures and a sparse decomposition for the so-called dynamic mode decomposition in~\cite{SparseKoopman}.

The central concept in this text is decoupling of dynamical systems into smaller subsystems. As subsystems, we consider ensembles of states of the dynamical system that are causally independent of the remaining states. This allows to treat these ensembles of states as separate dynamical systems, provided the constraint set satisfies certain compatibility conditions. This results in computational time reduction and builds on the work~\cite{Chen}. Even though our main goal is to exploit such decoupling computationally, we study the sparse structure at a rather general level, allowing for our results to be used within other computational frameworks and for other problems than those encountered in this work. The main novelty is the following: (i)  We generalize the method of \cite{Chen} to far more general sparse structures. (ii) We treat different problems than \cite{Chen}, namely in addition to ROA treated in \cite{Chen} we also treat the computation of the MPI set and GA. (iii) We show that any method for approximating the ROA, the MPI set, and GA with certain convergence properties allows for a reduction to lower dimensional systems such that convergence is preserved. (iv) We use the proposed decoupling scheme within the moment sum-of-squares hierarchy framework, obtaining a sparse computational scheme for the ROA, MPI set and GA with a guaranteed convergence from the outside to the sets of interest; to the best of our knowledge, this is the first time sparsity is exploited in the moment-sos hierarchy for dynamical systems without compromising convergence.

To determine the subsystems we represent the interconnection between the dynamics of the states by the directed sparsity graph of the dynamics. Informally, our main result reads: The dynamical system can be decomposed into subsystems that can be easily inferred from the sparsity graph. Further, this decomposition gives rise to decompositions of the ROA, MPI set, and GA.

In relation to existing works on decomposition of subsystems, the notion introduced in this work provides an \textit{exact} decomposition, i.e., no error is induced and the whole system can be uniquely recovered from subsystems. As a consequence, any preferred decomposition method for dynamical systems can be combined with our method as a refinement by applying it to any subsystem. From the existing decomposition methods that can be combined with ours, we mention \cite{Anderson} for a survey, \cite{AlMaruf} for decompositions based on comparison systems, \cite{Elkin} for decompositions based on a different notion on subsystems or \cite{Mischaikow} for topological decomposition, to name only a few existing methods.

We only consider continuous-time dynamical systems but all the results translate to discrete-time systems as well. Indeed, both the decoupling into subsystems and of the ROA, MPI and GA as well as the specific SOS approach have discrete-time analogs.

\subsection{Applications}

Our work is based on \cite{Chen}, where a 6D Quadrotor was used to illustrate the application of sparsity to reachable set computation. Another example from robotics involving sparsity can be found in control of a monocycle robot~\cite{MonoCycle}. Additional instances of sparse systems can be found in network systems, among which are energy networks, interaction networks, hierarchical networks, citation networks, the internet, social networks~\cite{Strogatz}; or logistic networks~\cite{Bullo}, (distributed) multicellular programming~\cite{CellarProgrammingR,CellarT}, and traffic networks~\cite{TrafficCascade2,TrafficCascade1}. These examples address important and large-scale applications. With increasing size of those networks, their analysis gets increasingly more important and challenging. Thus, exploiting the structure of these networks becomes necessary. One such structure that sometimes can be found in the mentioned networks is sparse causality. Causality describes the flow of information or the dependence between different states in such networks. In the case of social networks, such sparsity can be found in the so-called ``social bubbles". Another source of sparsity is ``locality". That is, when interactions take place physically, the geographical location influences the flow of information leading to a flow of information along continent $\rightarrow$ country $\rightarrow$ department $\rightarrow$ city $\rightarrow$ family/company/school/social clubs etc. 

\section{Notation}
The natural numbers are with zero included and denoted by $\N$. For a subset $J \subset \N$ we denote by $|J|$ its cardinality. The non-negative real numbers $[0,\infty)$ are denoted by $\R_+$. For two sets $K_1,K_2$ we denote their symmetric difference given by $(K_1 \setminus K_2) \cup (K_2 \setminus K_1)$ by $K_1 \Delta K_2$. The function $\dist(\cdot,K)$ denotes the distance function to $K$, and $\dist(K_1,K_2)$ denotes the Hausdorff distance of two subsets of $\R^n$ (with respect to a given metric or norm). We will denote the open ball centered at the origin of radius $r$ with respect to the Euclidean metric by $B_r(0)$. For an index set $J \subset \{1,\ldots,n\}$ we use the notation $\Pi_J$ for the projection $\Pi_J(x_1,\ldots,x_n) := (x_j)_{j\in J}$. The space of continuous functions on a set $X$ is denoted by $\Cc(X)$ and the space of continuously differentiable functions on $\R^n$ by $\Cc^1(\R^n)$.
The Lebesgue measure will always be denoted by $\lambda$. The ring of multivariate polynomials in variables $x = (x_1,\ldots,x_n)$ is denoted by $\R[x] = \R[x_1,\ldots,x_n]$ and for $k \in \N$ the ring of multivariate polynomials of total degree at most $k$ is denoted by $\R[x]_k$.

\section{Subsystem decomposition}
We consider a nonlinear dynamical system
\begin{equation}\label{eq:x'=f(x)}
    \dot{x}(t) = f(x(t))
\end{equation}
with the state $x \in \R^n$ and a locally Lipschitz vector field $f:\R^n\to\R^n$.

\paragraph*{Subsystems}
A central object in this text is the notion of subsystems of a dynamical system (\ref{eq:x'=f(x)}). We define a subsystem as follows.

\begin{definition}[Subsystem]\label{def:Subsystem}
    For a dynamical system~(\ref{eq:x'=f(x)}) and $I \subset \{1,\ldots,n\}$ we call a set of states $x_I := (x_{i})_{i \in I}$ a subsystem if $f_I := (f_i)_{i \in I}$ depends only on $x_I$. In that case, we say $I$ induces a subsystem.
\end{definition}
If $x_I$ is a subsystem, we can regard $f_I$ as a mapping from $\R^{|I|} \to \R^{|I|}$ and the dynamics of $x_I$ is given by
\begin{equation}\label{eq:subsystem}
\dot{x}_I(t) = f_I(x_I(t)).
\end{equation}
A similar relation holds for the respective flows. Denoting $\varphi_t:\R^n\to\R^n$ the flow of~\eqref{eq:x'=f(x)} and $\varphi_t^I:\R^{|I|}\to\R^{|I|}$ the flow of~\eqref{eq:subsystem}, we get
\begin{equation}\label{eq:FlowCommutePI}
    \varphi_t^I \circ \Pi_I = \Pi_I\circ \varphi_t,
\end{equation}
where $\Pi_I:\R^n\to\R^{|I|}$ denotes the canonical projection onto $x_I$, i.e., $\Pi_I(x) = x_I$.

\begin{remark}[Causality]\label{Rem:Causality}
    The notion of a subsystem is closely related to the concept of causality \cite{Granger,Peters,Arrow}. Namely, a set $I \subset \{1,\ldots,n\}$ induces a subsystem if and only if there is no causal influence on the states indexed by $I$, by the remaining states not indexed by $I$.
\end{remark}

\begin{definition}[Subsystem decomposition] The subsystems $x_{I_k}$, $k \in \{1,\ldots,m\}$, $m \le n$, form a subsystem decomposition of~\eqref{eq:x'=f(x)} if $\cup_{k=1}^m I_k = \{1,\ldots,n\}$.
\end{definition}
Note that the decomposition does not need to be disjoint, i.e., we allow $I_k\cap I_j \ne\emptyset$. Note also that the decomposition does not need to be unique (see Example~\ref{ex:subsysems}).

\paragraph*{Sparsity graph} We now introduce the sparsity graph of the dynamics $f$ (which is sometimes also called inference diagram) that allows one to determine an optimal subsystem decomposition.

\begin{definition}[Sparsity graph]\label{def:sparsitygraph}
    The sparsity graph $G_f$ for a function $f:\R^n\to\R^n$ is defined by:
    \begin{enumerate}
        \item The set of nodes is $(x_1,\ldots,x_n)$.
        \item $(x_i,x_j)$ is an edge if the function $f_j$ depends on $x_i$.
    \end{enumerate}
\end{definition}

\begin{remark}\label{rem:sparsitygraph}
    The sparsity graph describes the dependence of the dynamics of a state on other states. More precisely, there exists a directed path from $i$ to $j$ in the sparsity graph of $f$ if and only if the dynamics of $x_j$ depend (indirectly via other states) on the state $x_i$.
\end{remark}

\begin{example}\label{ex:subsysems}
As an example consider the following function $f:\R^{8} \rightarrow \R^{8}$
\begin{eqnarray}\label{eq:Examplef}
    f(x_1,\ldots,x_{8}) & = & (-x_1,x_1x_2,x_3x_1+x_4,x_3x_4,x_5^2+x_1,  \nonumber  \\
    & & x_2x_5-x_8,x_6,x_7 + x_8^2). 
\end{eqnarray}
\noindent
The sparsity graph of $f$ is shown in Figure \ref{fig:ExampleSparsityGraph}. The subsystems are $(x_1)$, $(x_1,x_2)$, $(x_1,x_5)$, $(x_1,x_2,x_5)$, $(x_1,x_3,x_4,x_5)$, $(x_1,x_2,x_5,x_6,x_7,x_8)$ and the whole (and empty) system. A subsystem decomposition is given by $(x_1,x_2,x_5,x_6,x_7,x_8)$, $(x_1,x_3,x_4,x_5)$.
\end{example}

\begin{figure}
	\begin{center}
	\includegraphics[width=0.6\textwidth]{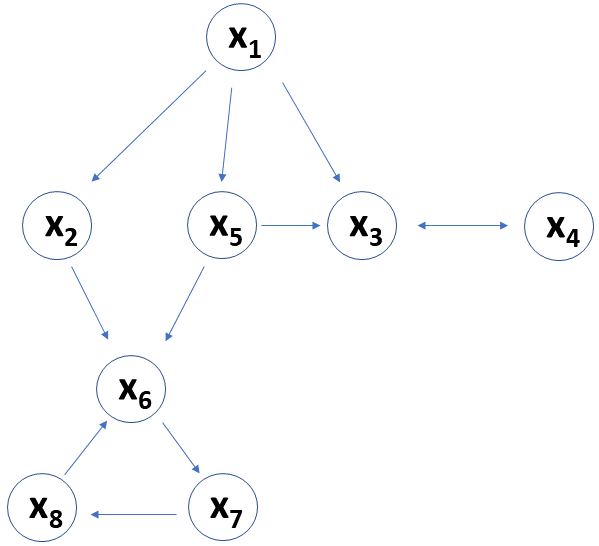}
	\caption{\footnotesize{Sparsity graph for the function $f$ from (\ref{eq:Examplef}).}}
\label{fig:ExampleSparsityGraph}
\end{center}
\end{figure}

From the sparsity graph, we infer the following important objects concerning subsystems.

\begin{definition}[Predecessor, leaf, root, past]\label{def:Past} $ $\vspace{-3.2mm}
\begin{enumerate}
\item For a directed graph with nodes $x_1,\ldots,x_n$ we call a node $x_i$ a predecessor of node $x_j$ if either $x_i = x_j$ or if there is a directed path from $x_i$ to $x_j$.
\item A node $x_i$ is called a leaf if it does not have a successor (i.e., all nodes connected to $x_i$ are its predecessors), and $x_i$ is called a root if it doesn't have a predecessor.
\item The set of all predecessors of $x_i$ is called the past of $x_i$ and denoted by $\mathbf{P}(x_i)$.
\end{enumerate}
\end{definition}

\begin{figure}
    \centering
    \includegraphics[width=0.6\textwidth]{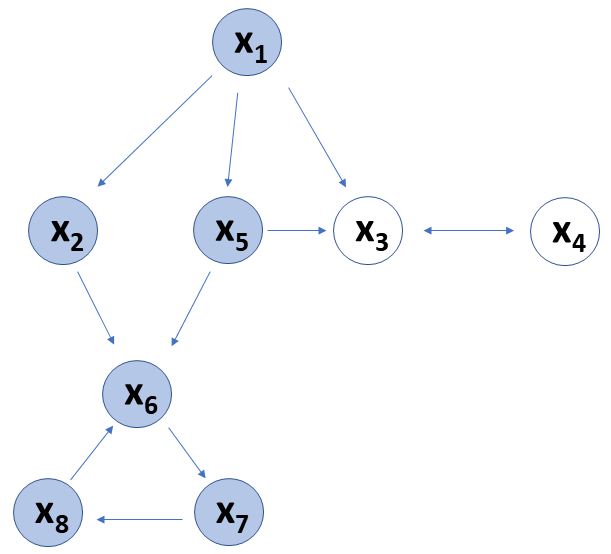}
    \caption{\footnotesize{The past of $x_6$ respectively $x_7$ respectively $x_8$ in the sparsity graph of $f$ from (\ref{eq:Examplef}) is colored in blue. The node $x_1$ is a root.}}\label{fig:sparseGraph}
\end{figure}

The following lemma states that the pasts of single nodes can be viewed as the building blocks of subsystems.
\begin{lemma}\label{lem:PastSubsystem}
    For a given node $x_i$ the past $\mathbf{P}(x_i)$ of this node determines the smallest subsystem containing $x_i$.
\end{lemma}

\begin{proof}
    First, we show that any subsystem $x_I$ containing $x_i$ has to contain $\mathbf{P}(x_i)$. By induction, it suffices to show that a subsystem $x_J$ containing $x_l$ has to contain all predecessors of $x_l$. Let $J$ induce a subsystem, $x_l$ with $l \in J$ be an arbitrary node and $x_j$ a predecessor of $x_l$. Let $x_j,x_{j_1},\ldots,x_{j_k},x_l$ be a directed path in the sparsity graph. It follows from the definition of a subsystem (Definition \ref{def:Subsystem}), that its immediate predecessor $x_{j_k}$ has to be contained in the subsystem induced by $J$. By induction, we conclude that each of the nodes $x_{j_{k-1}},x_{j_{k-2}},\ldots,x_{j_1},x_j$ are contained in the subsystem induced by $J$. It remains to show that $\mathbf{P}(x_i)$ is a subsystem. Let $I := \{j: x_j \in \mathbf{P}(x_i)\} \subset \{1,\ldots,n\}$ denote the set of indices of the past of $x_i$. Let us assume $\mathbf{P}(x_i)$ is not a subsystem, i.e. $f_I$ depends on a state $x_j \notin \mathbf{P}(x_i)$. Let $l \in I$ such that $f_l$ depends on such a state $x_j$. Then $x_j$ is a predecessor of $x_l$ and hence a predecessor of $x_i$.
    That means $x_j \in \mathbf{P}(x_i)$, which contradicts $x_j \notin \mathbf{P}(x_i)$.
\end{proof}

In acyclic sparsity graphs, the nodes with maximal past are leaves since a successor has a larger past than its predecessor. This observation will be essential in the following section where we address finding minimal subsystem decomposition.

\subsection{Optimal subsystem decomposition}

Choosing a subsystem decomposition that allows a decoupling into subsystems of small dimensions is key to maximizing computational savings and is the topic of this section.

A good choice of subsystems depends on both the dynamical interaction of nodes -- which is represented by the sparsity graph of the dynamics -- and the sparse structure of the constraint set. We begin with the dynamics.

A subset of nodes $V \subset \{x_1,\ldots,x_n\}$ is called strongly connected in the sparsity graph if for any two disjoint nodes $x_i,x_j \in V$ there exists a directed path from $x_i$ to $x_j$. We call $V$ a strongly connected component if $V$ is strongly connected and cannot be extended to a strictly larger strongly connected set. For a strongly connected component $V$, for each $x_i \in V$ its past $\mathbf{P}(x_i)$ contains $V$. Because $\mathbf{P}(x_i)$ is the smallest subsystem that contains $x_i$ we see that all nodes in $V$ appear always together in subsystems. For finding subsystems we can therefore turn to a reduced graph, the \emph{condensation graph} of the sparsity graph. In the condensation graph, each strongly connected component is contracted to one node. To give a precise meaning of contracting nodes we recall the notion of node contraction from graph theory~\cite[Section 1.7]{Diestel}.

\begin{definition}[Node contraction]\label{def:NodeContraction}
    Let $G = (V,E)$ be a (directed) graph with nodes $V$ and edges $E$. Let $W_1,\ldots,W_k \subset V$ be pairwise disjoint sets of nodes and fix elements $\omega_i \in W_i$ for $i = 1,\ldots,k$. The graph $G' = (V',E')$ obtained by contracting $W_1,\ldots,W_k$ is given by the nodes $V' :=  \left(V\setminus \bigcup\limits_{i = 1}^k W_i\right) \cup \{\omega_1,\ldots,\omega_k\}$ and for $x_j,x_l \in V'$ the pair $(x_j,x_l)$ is an edge if one of the following holds:
        \begin{enumerate}
            \item $x_j,x_l \in V$ and $(x_j,x_l) \in E$
            \item $x_l \in V$ and $x_j = \omega_i$ for some $1\leq i \leq k$ and there exists $x_m \in W_i$ such that $(x_m,x_l) \in E$
            \item $x_j \in V$ and $x_l = \omega_i$ for some $1\leq i \leq k$ and there exists $x_m \in W_i$ with $(x_j,x_m) \in E$.
        \end{enumerate}
\end{definition}

Figure~\ref{fig:NodeContraction} illustrates a node contraction for the graph from Figure~\ref{fig:sparseGraph}.

\begin{figure}[h]
    \centering
    \includegraphics[width = 0.95\textwidth]{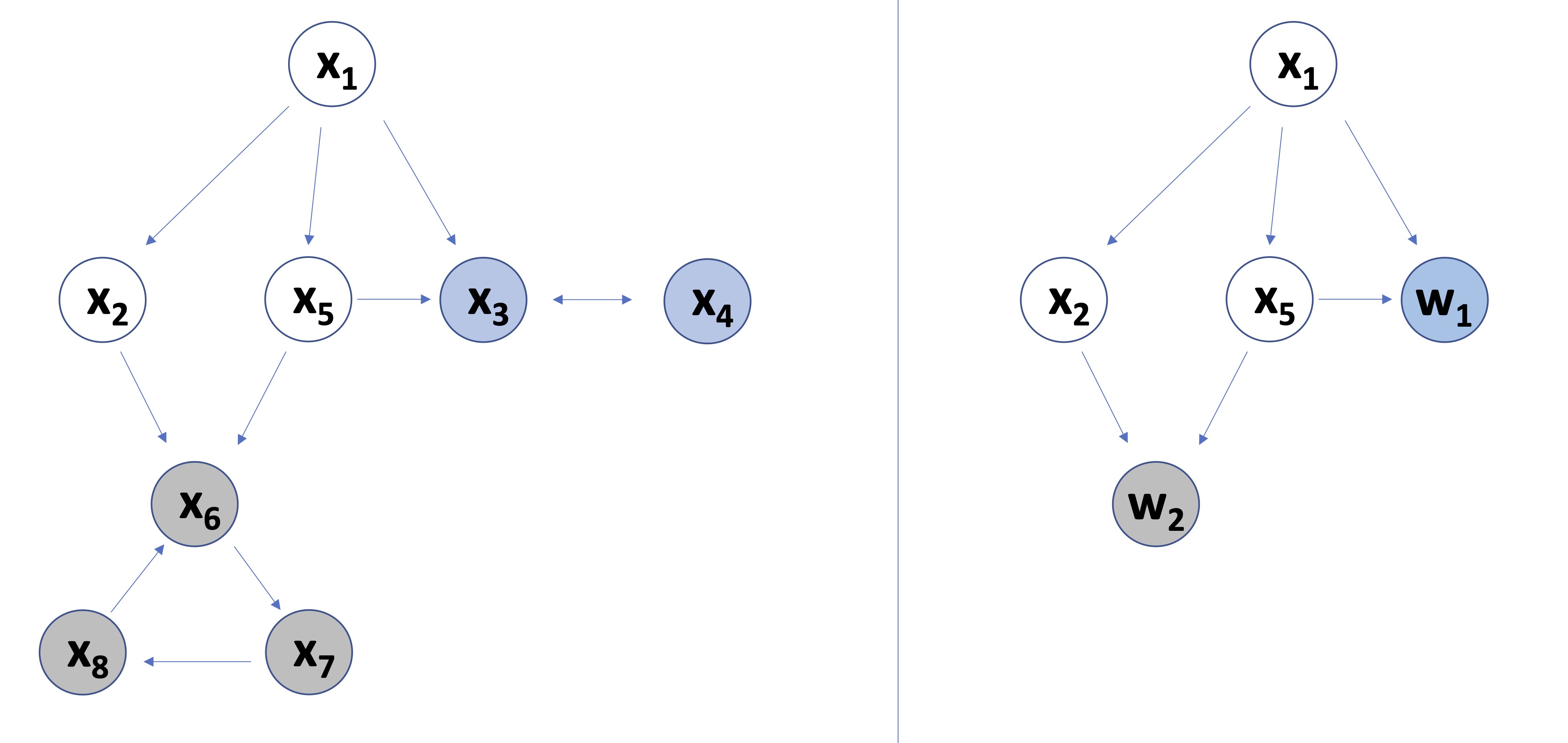}
    \caption{\footnotesize{Example of a node contraction (right) of the graph from Figure \ref{fig:sparseGraph} (left). Nodes $x_3,x_4$ in light blue are contracted to $w_1$ and nodes $x_6,x_7,x_8$ in light grey are contracted to $w_2$.}}
    \label{fig:NodeContraction}
\end{figure}

In the following, we will leverage the condensation graph of a sparsity graph in order to find subsystems by simple graph algorithms.

\begin{definition}[Condensed sparsity graph]\label{def:CondensationGraph}
    Let $G_f$ be the sparsity graph of a function $f:\R^n\to\R^n$. Let $W_1,\ldots,W_k$ be the strongly connected components of $G_f$. The condensed sparsity graph $W_f$ of $G_f$ is obtained by contracting the $W_1,\ldots,W_k$ in $G_f$ in the sense of Definition \ref{def:NodeContraction}.
\end{definition}

\begin{remark}\label{rem:ComputationCondensedGraph}
    The condensation graph can be computed by a depth first search (see~\cite[Section 22.5]{IntroductionToAlgorithms}) in $\mathcal{O}(n + m)$ where $m$ denotes the number of edges in the sparsity graph of $f$.
\end{remark}
\begin{figure}[!t]
\begin{center}
\includegraphics[width=0.4\textwidth]{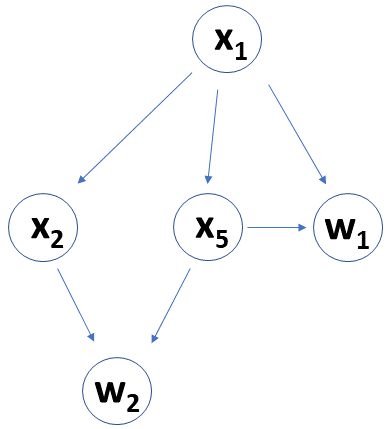}
\caption{\footnotesize{Condensation graph for the sparsity graph from Figure \ref{fig:ExampleSparsityGraph} of $f$ from (\ref{eq:Examplef}). The graph is identical with the node contraction graph on the right in Figure~\ref{fig:NodeContraction} because the contracted nodes in Figure~\ref{fig:NodeContraction} are exactly the connected components.}}
\label{fig:circle}
\end{center}
\end{figure}

The condensed sparsity graph is a directed forest, i.e. it does not have directed cycles and can have multiple roots (see Definition \ref{def:Past}). With regard to our notion of sparsity, it is the simplest graph representation of causal dependence in the dynamical system and builds the starting point for finding subsystems.

The smallest subsystems correspond to the roots in the condensed sparsity graph. The sparsest subsystem decomposition is obtained from the pasts of the leaves in the condensed sparsity graph; we specify this decomposition in Proposition \ref{rem:SubsystemDecompLeafs}.

\begin{proposition}\label{rem:SubsystemDecompLeafs}
    Let $x_{j_1},\ldots,x_{j_m}$ be the leaves in the condensed sparsity graph (where $W_1,\ldots,W_k \subset \{x_1,\ldots,x_n\}$ have been contracted to $\omega_1,\ldots,\omega_k$). The pasts $\mathbf{P}(x_{j_1})$,$\ldots$, $\mathbf{P}(x_{j_m})$ form a subsystem decomposition (where for $1\leq i\leq k$ each $\omega_i$ is replaced by the nodes $W_i$ it represents).
\end{proposition}

To prove Proposition \ref{rem:SubsystemDecompLeafs}, we first present its graph theoretical analogue.

\begin{lemma}\label{rem:Leafs}
	Any directed graph without cycles has at least one leaf. Furthermore, for directed graphs without cycles we have for the set $V$ of nodes that $V = \bigcup\limits_{x \text{ leaf}} \mathbf{P}(x)$.
\end{lemma}

\begin{proof}
    Let $W$ be a maximal path in the graph, i.e. a path that cannot be extended in $G$. Let $x$ be the last node in $W$. We claim that $x$ is a 
    leaf. If $x$ is not a leaf then there exists an edge $(x,y)$ in $G$ for some node $y$. By maximality of $W$ we cannot add $y$ to $W$, which means the edge $(x,y)$ has been used before in $W$. This means that $W$ has visited $x$ before, i.e. there is a part of $W$ that connects $x$ to itself, i.e. a cycle -- contradiction. For the remaining statement let $y$ be an arbitrary node. We can choose a longest path containing this node which has to end in a leaf $x$, hence $y$ is contained in the past of $x$.
\end{proof}

Now we prove Proposition \ref{rem:SubsystemDecompLeafs}.

\begin{proof}
    \textit{Of Proposition \ref{rem:SubsystemDecompLeafs}.} By Lemma \ref{lem:PastSubsystem}, the pasts $\mathbf{P}(x_{j_1}),\ldots,\mathbf{P}(x_{j_m})$ are all subsystems. It remains to show that
    \begin{equation}\label{eq:pastcover}
        \{x_i: i = 1,\ldots,n\} = \bigcup\limits_{l = 1}^m \mathbf{P}(x_{j_l}).
    \end{equation}
    Let $1\leq i \leq n$ and $W_1,\ldots,W_k$ be the strongly connected components in the sparsity graph. If $x_i \notin W_1 \cup \cdots \cup W_k$ then $x_i$ is a node in the condensed sparsity graph. By Lemma \ref{rem:Leafs}, $x_i \in \bigcup\limits_{l = 1}^m \mathbf{P}(x_{j_l})$. If $x_i \in W_s$ for some $1\leq s \leq k$ then $x_i \in \mathbf{P}(\omega_s)$ because $W_s$ is strongly connected. Since $\omega_s$ is a node in the condensed sparsity graph, we have, as above, $\omega_s \in \bigcup\limits_{l = 1}^m \mathbf{P}(x_{j_l})$, i.e. $\omega_s \in \mathbf{P}(x_{j_r})$ for some $1\leq r \leq m$. From $x_i \in \mathbf{P}(\omega_s)$, $\omega_s \in \mathbf{P}(x_{j_r})$, and the definition of the past we infer $x_i \in \mathbf{P}(x_{j_r})$. This concludes the statement.
\end{proof}

\subsection{Subsystem decomposition with constraints}
Often a dynamical system is restricted to a subset $X \subset \R^n$, i.e. solutions $x(\cdot)$ that do not stay in $X$ for all $t$ are deemed undesirable.
For our sparse approach, the constraint set $X$ has to decompose in accordance with the subsystem decomposition. We now define this rigorously.

\begin{definition}\label{def:XdecomposesAccordingly}
    We say that the constraint set $X$ decomposes according to a family $\mathcal{J}$ of index sets $J_1,\ldots,J_N \subset \{1,\ldots,n\}$ if there exist sets $X_1 \subset \R^{|J_1|},\ldots,X_N \subset \R^{|J_N|}$ such that
\begin{equation}\label{eq:XDecompositionSubsystemCovering}
    X = \{x \in \R^n: \Pi_{J_l}(x) \in X_{l} \text{ for } l = 1,\ldots,N\}.
\end{equation}
\end{definition}

The notion of a set $X$ decomposing according to a family of index sets generalizes the situation where the set $X$ factors into a Cartesian product. A factorization of $X$ into a Cartesian product corresponds to the case where the sets $J_1,\ldots,J_N$ are pairwise disjoint, as we mention in the following Remark. 

\begin{remark}\label{rem:Xfactors}
    In the case where $X$ factors into a Cartesian product
    \begin{equation*}
        X = X_1 \times \cdots \times X_N \text{ with } X_l \subset \R^{n_l} \text{ for } l = 1,\ldots,N
    \end{equation*}
    the set $X$ decomposes according to $J_1,\ldots,J_N$ with $J_1 := \{1,\ldots,n_1\}$ and $J_l = \{n_1 + \ldots + n_{l-1} +1, n_1 + \ldots n_l\}$ for $l = 2,\ldots,N$, namely $X = \{x \in \R^n: \Pi_{J_l}(x) \in X_l \text{ for } l = 1,\ldots,N\}$. 
\end{remark}





Based on a decomposition of $X$ in the sense of Definition \ref{def:XdecomposesAccordingly}, we can search for a subsystem decomposition for which $X$ decomposes accordingly. To do so we extend the sparsity graph $G_f$ of $f$ to the sparsity graph $G_{f,\mathcal{J}}$ which additionally includes the sparsity structure of $X$.

\begin{definition}[The sparsity graph $G_{f,\mathcal{J}}$]\label{def:G_f,I}
Let $G_f$ be the sparsity graph of $f$ and $\mathcal{J}$ a family of index sets $J_1,\ldots,J_N \subset \{1,\ldots,n\}$ for which $X$ decomposes accordingly. The graph $G_{f,\mathcal{J}}$ has the nodes $x_1,\ldots,x_n$ and $(x_i,x_l)$ is an edge of $G_{f,\mathcal{J}}$ if $(x_i,x_l)$ is an edge in $G_f$ or if there exists $1\leq r\leq N$ such that $i,l \in J_r$.
\end{definition}

For the graph $G_{f,\mathcal{J}}$ we form the condensation graph and the subsystem decomposition corresponding to the family of pasts of its leaves from Proposition \ref{rem:SubsystemDecompLeafs}. We will show in Theorem \ref{thm:mainthmformal}, that this indeed forms a subsystem decomposition for which $X$ decomposes accordingly. In Algorithm \ref{Alg:SubsysDecomposition} we state how to obtain this subsystem decomposition.

\begin{Algorithm}\label{Alg:SubsysDecomposition}
    Input: The sparsity graph $G_f$ of $f$ and a family $\mathcal{J}$ of index sets for which $X$ decomposes accordingly.
    \begin{enumerate}
        \item[1.] Compute the graph $G_{f,\mathcal{J}}$ from Definition \ref{def:G_f,I}.
        \item[2.] Compute the strongly connected components $W_1,\ldots,W_k$ of $G_{f,\mathcal{J}}$.
        \item[3.] Build the condensed graph $\mathcal{G}_{f,\mathcal{J}}$ of $G_{f,\mathcal{J}}$:  Contract each of the strongly connected components $W_i$ to $\omega_i$ (see Definition \ref{def:NodeContraction}).
        \item[4.] Find the leaves $v_1,\ldots,v_m$ in $\mathcal{G}_{f,\mathcal{J}}$.
        \item[5.] Define $I_1,\ldots,I_m$: Set $I_l := \mathbf{P}(v_l)$ for $l =1,\ldots,m$ (where for $1\leq i\leq k$ each $\omega_i$ is replaced by the nodes $W_i$ it represents).
    \end{enumerate}
    Output: Subsystem covering $x_{I_1},\ldots,x_{I_m}$.
\end{Algorithm}

The largest appearing state space dimension based on the subsystem decomposition obtained from Algorithm \ref{Alg:SubsysDecomposition} is given by the largest number of predecessors of a node in the graph $G_{f,\mathcal{J}}$. This number is given by
\begin{equation}\label{eq:LargestWeightedPast}
	\omega := \max\limits_{l} |\mathbf{P}(v_l)|.
\end{equation}
where $|\mathbf{P}(v_l)|$ is the number predecessors of the node $v_j$ in the sparsity graph $G_{f,\mathcal{J}}$. We formulate this statement precisely in the following theorem.

\begin{theorem}\label{thm:mainthmformal}
    Let $\mathcal{J}$ be a family of index sets for which $X$ decomposes accordingly. Algorithm \ref{Alg:SubsysDecomposition} gives a subsystem decomposition $x_{I_1},\ldots,x_{I_m}$ for which $X$ decomposes accordingly such that the largest subsystem contains $\omega$ nodes, i.e.
    \begin{equation*}
        \max\limits_{l = 1,\ldots,m} |I_l| = \omega,
    \end{equation*}
    where $\omega$ is given by (\ref{eq:LargestWeightedPast}).
\end{theorem}

\begin{proof}
    We will show that for $x_{I_1},\ldots,x_{I_m}$ it holds:
    \begin{enumerate}
        \item \label{itemFamilyOfLeavesSubsystemDecomposition} $x_{I_1},\ldots,x_{I_m}$ is a subsystem decomposition for which $X$ factors accordingly, and
        \item \label{itemFamilyLeavesNumberVariables} the largest number of variables in each of these subsystems is at most $\omega$.
    \end{enumerate}
    That $x_{I_1},\ldots,x_{I_m}$ is a subsystem decomposition can be verified similarly to Proposition \ref{rem:SubsystemDecompLeafs} because the sparsity graph $G_{f}$ is a subgraph of $G_{f,\mathcal{J}}$. To show \ref{itemFamilyOfLeavesSubsystemDecomposition}. it remains to check that $X$ decomposes accordingly. For $1\leq r \leq N$ and $i,j \in J_r$, the two nodes $x_i$ and $x_j$ form a cycle in the graph $G_{f,\mathcal{J}}$. Therefore, the nodes $(x_j)_{j \in J_r}$ get contracted in the condensation graph of $G_{f,\mathcal{J}}$ for all $r = 1,\ldots,N$ and each of the sets $I_l$ can be written as
    \begin{equation*}
        I_l = \bigcup\limits_{l \in Z_k} J_l,
    \end{equation*}
    where the sets $Z_1,\ldots,Z_m$ form a partition of $\{1,\ldots,N\}$. We claim that $X$ decomposes according to the subsystem decomposition $x_{I_1},\ldots,x_{I_m}$ by
    \begin{equation}\label{eq:XdecomposesAccordinglyPastDecomposition}
        X = \{x \in \R^n: \Pi_{I_k}(x) \in Y_k, k = 1,\ldots,m\}
    \end{equation}
    where for $k = 1,\ldots,m$ the sets $Y_k$ are given by
    \begin{equation*}
        Y_k := \{\Pi_{I_k}(x) \in \R^{|I_k|} : x \in \R^n, x_{J_l} \in X_l, l \in Z_k\}.
    \end{equation*}
    We conclude (\ref{eq:XdecomposesAccordinglyPastDecomposition}) via
    \begin{align*}
        X &= \{x \in \R^n: \Pi_{J_l}(x) \in X_l, l = 1,\ldots,N\}\\
        &= \{x \in \R^n: \Pi_{J_l}(x) \in X_l, l \in \bigcup\limits_{k = 1}^m Z_k\}\\
        &= \{x \in \R^n: \Pi_{I_k}(x) \in Y_k, k = 1,\ldots,m\},
    \end{align*}
    because $X$ decomposes according to $J_1,\ldots,J_N$. To verify claim (\ref{itemFamilyLeavesNumberVariables}) recall that each subsystem $x_{I_k}$ corresponds to the past of a node $v_k$ in the condensation graph of $G_{f,\mathcal{J}}$. Because the condensation graph contracts only the strongly connected components it holds $x_{I_k}$ is nothing else than the past in the sparsity graph $G_{f,\mathcal{J}}$ of the node $x_{i_k}$ with $x_{i_k} = v_k$.
\end{proof}

The subsystem decomposition from Algorithm \ref{Alg:SubsysDecomposition} is obtained by pasts of single nodes. Thus they are minimal in the sense that we cannot find a subsystem decomposition, based on the given decomposition $\mathcal{J}$ for $X$, such that all subsystems have strictly lower dimensions than $\omega$ from (\ref{eq:LargestWeightedPast}). But it remains to address finding a good decomposition of $X$. In applications, the sparsity of the dynamics often comes with the same sparse structure for $X$, i.e. $X$ decomposes according to any subsystem decomposition. But this might not always be the case. Finding the best decomposition of $X$, i.e. a decomposition that minimizes $\omega$ from (\ref{eq:LargestWeightedPast}), is a delicate task. We address a (suboptimal) procedure based on factoring $X$ into a Cartesian product motivated from Remark \ref{rem:Xfactors}. We say that index sets $J_1,\ldots,J_N \subset \{1,\ldots,n\}$ induce a factorization of $X$ if $J_1,\ldots,J_N$ is a partition of $\{1,\ldots,n\}$ and
\begin{equation*}
    X = \{x \in \R^n: \Pi_{J_l}(x) \in \Pi_{J_l}(X) \text{ for } l = 1,\ldots,N\}.
\end{equation*}
As noted in Remark \ref{rem:Xfactors}, $X$ decomposes according to the family $\mathcal{J}$ of index sets $J_1,\ldots,J_N$ if $J_1,\ldots,J_N$ induce a factorization. We propose using a minimal such factorization $\mathcal{J}$ as input to Algorithm \ref{Alg:SubsysDecomposition} to obtain a fine subsystem decomposition. The following lemma assures existence of a minimal factorization.

\begin{lemma}\label{lem:MinimalFactorization}
    There exists a minimal factorization for $X$; that is there exist index sets $J_1,\ldots,J_N$ that induce a factorization of $X$, such that for any other factorization induced by $\tilde{J}_1,\ldots,\tilde{J}_{\tilde{N}}$ we have for all $l = 1,\ldots,\tilde{N}$ that $\tilde{J}_l = \bigcup\limits_{k: J_k \subset \tilde{J}_l} J_k$.
\end{lemma}

\begin{proof} We give a proof in the Appendix.
\end{proof}

\section{Decomposition of the region of attraction, maximum positively invariant set, and global attractors}\label{sec:genGraphs}
In this section, we describe decompositions of the region of attraction (ROA), maximum positively invariant set (MPI) set, and the global attractors (GA) based on subsystems. Given subsystems induced by $I_1,\ldots,I_k$, the guiding idea is to characterize the desired set $S$ for the whole system via the corresponding sets $S_1,\ldots,S_m$ for the subsystems. We show that the set $S$ decomposes according to $I_1,\ldots,I_m$ by
\begin{equation}\label{eq:DecompositionS}
     S = \{x \in X: \Pi_{I_j}(x) \in S_j \text{ for } j = 1,\ldots,m\}.
\end{equation}
This decomposition shows that we can find the desired set $S$ by computing its analogues for the subsystems and ``gluing" them together as in (\ref{eq:DecompositionS}). We illustrate a computational application of this approach in Section \ref{sec:StructuredSemiDefProgramming}.
    
The sets of interest, the ROA, MPI set and GA, are defined in the following.

\begin{definition}[Region of attraction]
	For a dynamical system, a finite time $T \in \R_+$ and a target set $X_\mathrm{T} \subset X$ the region of attraction (ROA) of $X_\T$ is defined as
	\begin{eqnarray}\label{def:RoA}
		R_T  := \big\{ x_0 \in X & : & \exists x(\cdot) \text{ s.t. } \dot{x}(t) = f(x(t)), \; x(0) = x_0, \notag \\
		& &  x(t) \in X \text{ on } [0,T],  \; x(T) \in X_\T\big\}.
	\end{eqnarray}
\end{definition}

\begin{remark} The reachable set from an initial set $X_\I \subset X$ in time $T$
\begin{eqnarray}\label{eq:ReachableSet}
	S_T := \{x \in X & : &  \dot{x}(t) = f(x(t)), x(t) \in X \text{ on } [0,T],\notag\\ & & x(T) = x, x(0) \in X_\mathrm{I}\}
\end{eqnarray}
can be obtained by time reversal, i.e. by $S_T = R_T$ for $X_\T := X_\I$ and the dynamics given by $\dot{x} = -f(x)$.
\end{remark}

\begin{definition}[Maximum positively invariant set]\label{def:MPI}
    The maximum positively invariant (MPI) set for a dynamical system (\ref{eq:x'=f(x)}) with constraint set $X$ is the set of initial conditions $x_0$ such that the solutions $\varphi_t(x_0)$ stay in $X$ for all $t \in \R_+$.
\end{definition}
The MPI set will be denoted by $M_+$ in the following.

\begin{definition}[Global attractor]
	A compact set $\A \subset X$ is called the global attractor (GA) if it is minimal uniformly attracting, i.e., it is the smallest compact set $\A$ such that
	\[
	\lim_{t\to\infty} \dist(\varphi_t(M_+),\A) = 0.
	\]
\end{definition}

\begin{remark}\label{rem:Attractor}
If $X$ is compact then the global attractor exists and is unique~\cite{Robinson}. Further, the global attractor is characterized by being invariant, i.e. $\varphi_t(\A) = \A$ for all $t \in \R_+$, and attractive see~\cite{Robinson}.
\end{remark}

In the following Theorem \ref{thm:DecouplingRoAMPIAttractorGeneral} we present how the sparsity in the dynamics and the state constraint $X$ give rise to a decomposition of the ROA, MPI set and GA as in (\ref{eq:DecompositionS}). The key point of the decomposition of $X$ is that we can certify membership of $x$ to $X$ through membership of $x_{I_j}$ to $X_j$ for $1\leq j \leq m$ and the last task is performed solely in the subsystem $x_{I_j}$. 

\begin{theorem}\label{thm:DecouplingRoAMPIAttractorGeneral}
    Let $x_{I_1},\ldots,x_{I_m}$ be a subsystem decomposition of (\ref{eq:x'=f(x)}), $X$ be compact and decompose according to $I_1,\ldots,I_m$ (see Definition \ref{def:XdecomposesAccordingly}). In the case of the ROA we also assume that there exist $X_{1,\mathrm{T}} \subset \R^{|I_1|},\ldots, X_{m,\mathrm{T}} \subset \R^{|I_m|}$ such that $X_\mathrm{T} = \{x \in \R^n: x_{I_j} \in X_{j,\mathrm{T}}, 1\leq j \leq m\}$ and let $T > 0$. 
    Then the ROA $R_T$, MPI set $M_+$ and GA $\A$ are given by
	\begin{equation}\label{eq:RoASparseGeneral}
		R_T = \{x \in X : \Pi_{I_j}(x) \in R_T^j \text{ for } j = 1,\ldots,m\}
	\end{equation}
	\vspace{-5mm}
	\begin{equation}\label{eq:MPISparseGeneral}
		M_+ = \{x \in X : \Pi_{I_j}(x) \in M_+^j \text{ for } j = 1,\ldots,m\}
	\end{equation}
	\begin{equation}\label{eq:AttractorSparseGeneral}
		\A = \{x \in X : \Pi_{I_j}(x) \in \A^j \text{ for } j = 1,\ldots,m\}
	\end{equation}
	where for $1\leq j \leq m$ the sets $R_T^j, M_+^j, \A^j$ denote the ROA, MPI set and GA for the subsystem $x_{I_j}$ with constraint set $X_j$.
\end{theorem}

\begin{proof} We start with the ROA. Let $R$ denote the right-hand side of (\ref{eq:RoASparseGeneral}). Let $x_0 \in R$. We have to show that for the solution $x(\cdot)$ of (\ref{eq:x'=f(x)}) with initial value $x_0$ we have $x(t) \in X$ for $t \in [0,T]$ and $x(T) \in X_T$. For $1\leq j \leq m$ from (\ref{eq:FlowCommutePI}) we infer that $x_{I_j}(\cdot)$ is the solution of the subsystem equation $\dot{x}_{I_j}(t) = f_{I_j}(x_{I_j}(t))$ with initial value $\Pi_{I_j}(x_0)$. Thus, by definition of the set $R$, it follows
\begin{equation}\label{eq:RInR_T}
    \Pi_{I_j}(x(t)) \in X_j \text{ for all } t \in [0,T] \text{ and } \Pi_{I_j}(x(T)) \in X_{j,\mathrm{T}}.
\end{equation}
Because $X$ and $X_{\mathrm{T}}$ decompose accordingly (\ref{eq:RInR_T}) means that $x(t) \in X$ for all $t \in [0,T]$ and $x(T) \in X_{\mathrm{T}}$ -- in other words $x_0 \in R_T$. To see that the ROA $R_T$ is contained in $R$ let $x_0 \in R_T$ and let $x(\cdot)$ be the corresponding solution of the dynamical system with initial value $x_0$. Then from $x(t) \in X$ for all $t \in [0,T]$ we get $\Pi_{I_j}(x(t)) \in X_j$ for all $1\leq j \leq m$ because $X$ decomposes accordingly, and $x(T) \in X_\mathrm{T}$, i.e. $\Pi_{I_j}(x(T)) \in X_{j,\mathrm{T}}$ for all $1\leq j \leq m$ because $X_\mathrm{T}$ decomposes accordingly. Again, by (\ref{eq:FlowCommutePI}), it follows $\Pi_{I_j}(x) \in R_T^j$ for $1\leq j \leq m$, and hence $x \in R$. For the MPI set $M_+$ the argument is similar. Let $M$ denote the right-hand side of (\ref{eq:MPISparseGeneral}). For $x_0 \in M$ let $x(\cdot)$ denote the solution of (\ref{eq:x'=f(x)}) with initial value $x_0$. By definition of $M$ it holds $\Pi_{I_j}(x(t)) \in X_j$ for all $t \in \R_+$. Since $X$ decomposes accordingly it follows $x(t) \in X$ for all $t \in \R_+$ and thus $x_0 = x(0) \in M_+$. On the other hand for $x_0 \in M_+$ it follows $x(t) \in X$ in for all $t \in \R_+$ where $x(t)$ is the solution of (\ref{eq:x'=f(x)}). Since $X$ decomposes correspondingly this means $\Pi_{I_j}(x(t)) \in X_j$ for all $1\leq j \leq m$ and $t \in \R_+$. By (\ref{eq:FlowCommutePI}) the function $\Pi_{I_j}(x(\cdot))$ is the solution for the subsystem equation $\dot{\Pi_{I_j} \circ x}(t) = f_{I_j}(\Pi_{I_j}(x(t)))$ with initial value $\Pi_{I_j}(x_0)$, which shows that $\Pi_{I_j}(x_0) \in M_+^j$ for each $1\leq k \leq m$ -- that just states that $x_0 \in M$. For the GA we use the result proven for the MPI set and that $\A = M_+ \cap M_-$ where $M_+$ denotes the MPI set and $M_-$ the maximum negatively invariant set. The MNI set is the MPI set in reversed time direction (see \cite{Robinson}) and hence the decoupling result is also true for the MNI set. We get
\begin{eqnarray*}
    \A & = & M_+ \cap M_-\\
      & = & \{x \in X : \Pi_{I_j}(x) \in M_+^j \text{ for } j = 1,\ldots,m\} \cap\\
      & & \{x \in X : \Pi_{I_j}(x) \in M_-^j \text{ for } j = 1,\ldots,m\}\\
      & = & \bigcap\limits_{j = 1}^m\{x \in X : \Pi_{I_j}(x) \in M_+^j \cap M_-^j = \A^j\}\\
      & = & \{x \in X : \Pi_{I_j}(x) \in \A^j \text{ for } j = 1,\ldots,m\}
\end{eqnarray*}
\end{proof}


\begin{remark}
    The sets $R_T^j$, $M_+^j$ and $\A^j$ for $1\leq j \leq m$ in Theorem \ref{thm:DecouplingRoAMPIAttractorGeneral} do not coincide with $\Pi_{I_j}(R_T)$, $\Pi_{I_j}(M_+)$ respectively $\Pi_{I_j}(\A)$ respectively. It is easy to construct examples where the sets $R_T$, $M_+$ and $\A$ are empty but the sets $R_T^j$, $M_+^j$ and $\A^j$ are not. We give a simple example for which $M_+ = \emptyset$ while $M_+^{j} \neq \emptyset$. Consider the system given by $\dot{x}_1 = \dot{x}_2 = 0$ and $\dot{x}_3 = 1$ on $X = [0,1]^3$. We have the subsystem decomposition $x_{I_1},x_{I_2},x_{I_3}$ given by $I_1 = \{1\}, I_2 = \{1,2\}, I_3 = \{1,3\}$ for which $X$ decomposes accordingly. Here $M_+^{3}$ and hence also $M_+$ are empty while $M_+^{1} = [0,1]$ and $M_+^{2} = [0,1]^2$ are not empty.
\end{remark}

Theorem \ref{thm:DecouplingRoAMPIAttractorGeneral} gives rise to the following algorithm for a sparse computation of the ROA, MPI set, and GA.

\begin{Algorithm}[Decoupling procedure]\label{Alg:DecouplingProcedure}
Input: A dynamical system induced by $f$ with constraint set $X$ (and terminal constraint $X_\mathrm{T}$ for the ROA) and a method for approximating/computing the ROA, MPI set, or GA for an arbitrary dynamical system.
\begin{enumerate}
    \item[1.] Subsystem selection: Find a subsystem decomposition $x_{I_1},\ldots,x_{I_m}$ such that $X$ (and $X_\mathrm{T}$) decomposes according to $I_1,\ldots
    ,I_m$.
    \item[2.] Compute approximations for subsystems: Use the given method to compute approximations $\hat{S}_{1},\ldots,\hat{S}_{m}$ of the ROAs, MPI sets or GAs respectively for the subsystems induced $I_1,\ldots,I_m$.
    \item[3.] Gluing: Glue $\hat{S}_1,\ldots,\hat{S}_m$ together as suggested by Theorem \ref{thm:DecouplingRoAMPIAttractorGeneral} by
    \begin{equation*}
        \hat{S} := \{x \in X: \Pi_{I_j}(x) \in \hat{S}_{j} \text{ for } j = 1,\ldots,m\}.
    \end{equation*}
\end{enumerate}
\end{Algorithm}

Next, we show that the decoupling procedure preserves the convergence of any method for constructing outer approximations to ROA, MPI, or GA, which is crucial for practical deployment.


To quantify convergence, we consider the following two (pseudo) metrics on subsets on $\R^n$:
\begin{enumerate}
    \item Hausdorff distance: For sets $K_1,K_2 \subset \R^n$ the Hausdorff distance $\dist(\cdot,\cdot)$ is defined by
\begin{equation*}
    \dist(K_1,K_2) := \inf \{\varepsilon \geq 0 : K_2 \subset B_\varepsilon(K_1), \; K_1 \subset B_\varepsilon(K_2)\}
\end{equation*}
where $B_\varepsilon(K_i) := \{x + z \in \R^n : x\in K_i, z\in B_\varepsilon(0)\}$ and $B_\varepsilon(0) \subset \R^n$ is the euclidean ball with radius $\varepsilon$.
\item Lebesgue measures discrepancy: For sets $K_1,K_2\subset \R^n$ the Lebesgue measure discrepancy $d_\lambda(K_1,K_2)$ is defined by
\begin{equation}\label{eq:defLebMeasureDiscrepancy}
    d_\lambda(K_1,K_2):= \lambda(K_1 \Delta K_2).
\end{equation}
\end{enumerate}

In the following theorem, we state a quantitative bound (with respect to $\dist$ and $d_\lambda$) for the approximation of a desired set by the gluing construction from point 3 in Algorithm \ref{Alg:DecouplingProcedure}. The only important property is that the set of interest $S$ decomposes according to $I_1,\ldots,I_m$. Therefore we formulate it independently from a dynamic context.

\begin{theorem}\label{thm:GeneralizedDecouplingProcedure}
Let $n \in \mathbb{N}$ and $I_1,\ldots,I_m \subset \{1,\ldots,n\}$ with $\bigcup_{k = 1}^m I_k = \{1,\ldots,n\}$. For $k = 1,\ldots,m$ let $S_k \subset \R^{|I_k|}$ and $S \subset \R^n$ for which it holds $S= \{x \in \R^n: \Pi_{I_k}(x)\in S_k,k\in\{1,\ldots,m\}\}$. Let $(S_1^{l})_{l \in \N},\ldots,(S_m^{l})_{l \in \N}$ be sequences of sets with $S_k^{l} \subset \R^{|I_k|}$ for all $1\leq k \leq m$ and $l\in \N$ and set
\begin{equation*}
         S^{l} = \{x \in \R^n:  \Pi_{I_k}(x) \in  S_k^{l},\,k\in\{1,\ldots,m\}\}.
    \end{equation*}
The following hold
\begin{enumerate}
    \item Suppose that $ S_k^{l}\supset S_k$ for all $1\leq k \leq m, l\in \N$ and
    \begin{equation}\label{eq:ApproxQualityHausdorffGeneral}
        \mathrm{dist}( S_k^{l},S_k) \to 0 \text{ as } l\to\infty.
    \end{equation}
    Then
    \begin{equation}\label{eq:HausdorffConvergenceGeneral}
            \dist(S^{l},S) \rightarrow 0\; , \; \text{ as } l \rightarrow \infty.
        \end{equation}
    and $ S^{l} \supset S$.
    \item For the Lebesgue measures discrepancy $d_\lambda$ it holds
        \begin{eqnarray}\label{eq:ApproxQualityLebesgueGeneral}
                d_\lambda(S,S^{l}) & \leq & \sum\limits_{k = 1}^m \lambda(S_k \Delta {S}^{l}_k) \lambda(\Pi_{\{1,\ldots,n\}\setminus I_k}(X)).
        \end{eqnarray}
        In particular, if ${S}_k^{l}$ converges to $S_k$ as $l \rightarrow \infty$ with respect to $d_\lambda$ for all $k = 1,\ldots,m$ then ${S}^{l}$ converges to $S$ with respect to $d_\lambda$.
\end{enumerate}
\end{theorem}

\begin{proof}
    For the first statement note that the inclusion $ S_k^{l}\supset S_k$ for all $1\leq k \leq m, l\in \N$ implies ${S} \supset S$. To show the claim (\ref{eq:HausdorffConvergenceGeneral}), let us assume it doesn't hold. Then there exists $\varepsilon > 0$ and an unbounded subsequence $(l_r)_{r \in \N}$ such that
\begin{equation}\label{eq:HausdorffDistanceEps}
    \dist({S}^{l_r},S) > \varepsilon
\end{equation}
and we find points $x^{(l_r)} \in {S}^{l_r}$ with $\dist(x^{(l_r)},S) > \varepsilon$. From boundedness of $S_1,\ldots,S_m$ and the assumption (\ref{eq:ApproxQualityHausdorffGeneral}) it follows that there exists $x \in \R^{n}$ and a subsequence of $(l_r)_{r \in \N}$ which we will still denote by $(l_r)_{r \in \N}$ such that $x^{(l_r)} \rightarrow x$ as $r \rightarrow \infty$. By assumption (\ref{eq:ApproxQualityHausdorffGeneral}) there exist $y^j_{l_r} \in S_j$ for $k = 1,\ldots,m$ with $\|y_k^{l_r} - x^{(l_r)}_{I_k}\| \rightarrow 0$ as $m \rightarrow \infty$. Hence also $y^{l_r}_k \rightarrow \Pi_{I_k}(x)$ as $m \rightarrow \infty$ for $k = 1,\ldots,m$. Because $S_1,\ldots,S_l$ are closed it follows $\Pi_{I_k}(x) \in S_k$ for $k = 1,\ldots,m$ and by Theorem \ref{thm:DecouplingRoAMPIAttractorGeneral} we get $x \in S$. In particular, we get
\begin{equation*}
    \varepsilon < \dist(x^{(l_r)},S) \leq \|x^{(l_r)} - x\| \rightarrow 0
\end{equation*}
as $m \rightarrow \infty$, which is a contradiction. To conclude (\ref{eq:ApproxQualityLebesgueGeneral}) note that
\begin{equation*}
    S \Delta {S} \subset \bigcup\limits_{k = 1}^m\{x \in X : \Pi_{I_k}(x) \in S_k\Delta {S}_k\}.
\end{equation*}
Applying the Lebesgue measure to this inclusion gives
\begin{align*}
    \lambda(S \Delta {S}^{l}) & \leq \sum\limits_{k = 1}^m \lambda\left(\{x \in X: \Pi_{I_k}(x) \in S_k\Delta {S}^{l}_k\}\right)\\
    & \leq \sum\limits_{k = 1}^m \lambda(S_{k} \Delta {S}^{l}_{k}) \lambda(\Pi_{\{1,\ldots,n\}\setminus I_k}(X)).
\end{align*}
\end{proof}

\begin{remark}\label{rem:ConvRates}
    Statement (2) in Theorem \ref{thm:GeneralizedDecouplingProcedure} shows that the approximation error of $S^l$ for $S$ is at most proportional to the sum of the approximation errors $S_k^l$ for $S_k$. For problems from dynamical systems, the quality of approximations $S_k^l$ of $S_k$ typically depends on the dimension of the system and gets weaker with increasing dimension.  Thus, convergence rates carry over to the sparse approach and benefit from lower dimensions.
\end{remark}

In Section \ref{sec:StructuredSemiDefProgramming} we apply the decomposition procedure to approximate  the ROA, MPI set and GA via the convex optimization techniques from~\cite{BossRoA},~\cite{Boss14} and~\cite{Us}.

\section{Sparse approximation via structured SDPs}\label{sec:StructuredSemiDefProgramming}

As an illustrative example, we pair our decomposition technique, Algorithm~\ref{Alg:DecouplingProcedure}, with the convex optimization methods for computing outer approximations of the  ROA, MPI and GA  from~\cite{BossRoA,Boss14,Us}. The computational methods from~\cite{BossRoA,Boss14,Us} come with strong theoretical guarantees (convexity and convergence) but scale unfavorably with respect to the state-space dimension. Therefore, dimension reduction techniques, such as our approach, are needed to scale these methods to high dimensions. Further, the outer approximations provided by these methods satisfy the conditions of Theorem \ref{thm:GeneralizedDecouplingProcedure}, and therefore the decoupling procedure is directly applicable while retaining convergence guarantees. In addition, we propose a further sparsification of the method that incorporates sparsity directly into the semidefinite programs. This leads to different outer approximations at a similar computation time at the cost of convergence guarantees. However, we show that the convergence can be restored by combining the two approaches.




\subsection{Linear program representations for the ROA, MPI and GA}\label{sec:LPs}

We start by recalling the method from~\cite{Boss14} for computing the  MPI set $M_+$ based on infinite-dimensional linear programming (LP) and refer to \cite{BossRoA,Us} for similar results for the ROA and the GA. We will pair this approach with subsystem decompositions in subsections \ref{subsec:MainThm} and \ref{SubSec:SparseImprovement}.

In~\cite{Boss14}, the following LP was proposed
\begin{align}\label{LP:MPIDual}
	\begin{tabular}{lcll}
		$d^* :=$ & $\inf\limits_{v,w}$ & $\int\limits_{X} w(x) \; d\lambda(x)$ &\\
		& s.t. & $v \in \Cc^1(\R^{n}), w \in \Cc(X)$ &\\
			 & & $ \nabla v \cdot f \leq v $ & on $X$\\
			 & & $w \geq 0$& on $X$\\
			 & & $w \geq v + 1$ & on $X$\\
	\end{tabular}
\end{align}

\begin{remark}[Interpretation of~(\ref{LP:MPIDual})]\label{rem:InterpretLPMPI}
An important property of the LP (\ref{LP:MPIDual}) is that for any feasible $(v,w)$ it holds
\begin{equation*}
    w^{-1}([1,\infty)) \supset M_+.    
\end{equation*}
This can be verified as follows~\cite{Boss14}: The constraint $\nabla v \cdot f\leq v$ implies $v(x) \geq 0$ on $M_+$, see \cite[Lemma 4]{Boss14}, and hence, by the last constraint in (\ref{LP:MPIDual}), $w(x) \geq v(x) +1 \geq 1$ on $M_+$. Further, in \cite{Boss14}, it was shown that the outer approximation $w^{-1}([1,\infty))$ of the MPI set gets tight; more precisely, any minimizing sequence $(w_m,v_m)$ of the LP (\ref{LP:MPIDual}) satisfies $w_m \rightarrow \chi_{M_+}$ pointwise almost everywhere as $m\rightarrow \infty$ where $\chi_{M_+}$ denotes the indicator function of $M_+$ and it holds
$$
d^* = \int\limits_X \chi_{M_+} (x) \; dx = \lambda(M_+)
$$
the volume of the MPI set $M_+$.
\end{remark}




\subsection{Semidefinite programs for the ROA, MPI set and GA} \label{Subsec:SDP}
In this section, we recall the approach for reducing the LP (\ref{LP:MPIDual}) to a hierarchy of semidefinite programs (SDPs) from \cite{Boss14}; for the ROA and the GA, we refer to~\cite{BossRoA,Us}. For this approach, it is necessary to assume algebraic structure of the problem. For this standard procedure, we refer to~\cite{LasserreMoments} or~\cite{LasserreBook} for details.

\begin{assumption} \label{Assumption:Semialgebraic} The vector field $f$ is polynomial and $X \subset \R^n$ is a compact basic semi-algebraic set, that is, there exist polynomials $p_1,\ldots,p_{i},q_1,\ldots,q_l \in \R[x]$ such that $X = \{x \in \R^{n}: p_j(x) \geq 0 \text{ for } j = 1,\ldots,i\}$. Further, we assume (without loss of generality) that one of the $p_j$ is of the form $R^2 - \|x\|_2^2$ for some  $R >0$.
\end{assumption}


We briefly outline the idea for the SDP tightening of the LP (\ref{LP:MPIDual}) and refer to \cite{henrionBook,Boss14,BossRoA,Us} for details.

\paragraph{The SOS hierarchy for~(\ref{LP:MPIDual})} Formulating the SOS hierarchy for the optimization problem~(\ref{LP:MPIDual}) consists of the following steps:
\begin{enumerate}
    \item Replacing the search space in~(\ref{LP:MPIDual}) by the space of polynomials. This is justified by the Stone-Weierstraß theorem and the existence of strictly feasible points.
    \item Truncating the degree of the polynomials: Bounding the total degree of the polynomials induces a sequence of finite dimensional optimization problems (in the coefficients of the polynomials)
    \item Tightening each non-negativity constraint to a sum-of-squares condition. This allows to represent the resulting optimization problem as an SDP~\cite{LasserreBook}
    \item Finally, convergence is guaranteed by Putinar's positivstellensatz~\cite{Putinar}.
\end{enumerate}
Recalling that for $k\in \N$ we denote by $\R[x]_k$ the space of polynomials of degree at most $k$, the above procedure induces the following SOS-hierarchy for~(\ref{LP:MPIDual}):

Let $d_f := \max\{\deg f_1,\ldots,\deg f_n\}$ be the maximal degree of the polynomial map $f = (f_1,\ldots,f_n)\in \R[x]^n$. For $k \in \N$, where $k$ is the maximal total degree of the occurring polynomials, we formulate the following SOS program
\begin{align}\label{SDP:MPIDual}
	\begin{tabular}{lcll}
		$d_k^* :=$ & $\inf\limits_{v,w}$ & $\int\limits_X w(x) \; d\lambda(x)$\\
		& s.t. & $v \in \R[x]_{k+1-d_f}, w \in  \R[x]_k$ &\\
			& & $ v -\nabla v \cdot f = s_1 + \sum\limits_{j = 1}^{i} a_j p_j$ &\\
			& & $w = s_2 + \sum\limits_{j = 1}^i b_j p_l$&\\
			& & $w - v - 1 = s_3 + \sum\limits_{j = 1}^i c_j p_j$ &
	\end{tabular}
\end{align}
with sum-of-squares polynomials $s_1,s_2,s_3,a_j,b_j,c_j \in \R[x]$, for $j = 1,\ldots,i$, such that all occurring polynomials in (\ref{SDP:MPIDual}) are of degree at most $k$. The cost term $\displaystyle\int\limits_X w(x) \; d\lambda(x)$ is linear in $w$ and can be evaluated (linearly) via the coefficients of $w$ if the moments $m_\alpha := \int\limits_X x^\alpha \; dx$ are known. More precisely, for $w \in \R[x]_k$ with $w = \sum\limits_{|\alpha|\leq k} w_\alpha x^\alpha$ we have
\begin{equation*}
    \int\limits_X w(x) \; d\lambda(x) = \sum\limits_{|\alpha|\leq k} w_\alpha m_\alpha.
\end{equation*}

By~\cite{Boss14}, the optimal values $d_k^*$ of (\ref{SDP:MPIDual}) converge monotonically from above to the Lebesgue measure of the MPI set. For optimal solutions $(v_k,w_k)$ of (\ref{SDP:MPIDual}), the sets \begin{equation}\label{eq:Superlevelset}
    S^{(k)}:= w_k^{-1}([1,\infty)] = \{x \in X: w_k(x) \geq 1\}
\end{equation}
are outer approximations that get tight (with respect to Lebesgue measure discrepancy) as $k \rightarrow \infty$.

\subsection{Applying Algorithm~\ref{Alg:DecouplingProcedure}}\label{subsec:MainThm}

Combining Theorem~\ref{thm:GeneralizedDecouplingProcedure} with the convergence properties for the SOS-hierarchy for~(\ref{LP:MPIDual}) from~\cite{Boss14} we get the following corollary.

\begin{corollary}\label{cor:SDPSparse}
    Algorithm \ref{Alg:DecouplingProcedure}, where in step 2 we use the SDP hierarchy (\ref{SDP:MPIDual}), produces converging outer approximations of the MPI set, i.e. $S^{(k)} \supset S$ for all $k \in \N$ and
\begin{equation*}
    d_\lambda(S^{(k)},S) = \lambda(S^{(k)} \Delta S) \rightarrow 0 \text{ as } k\rightarrow \infty
\end{equation*}
    where $S$ denotes the MPI set for the whole system. Using Algorithm \ref{Alg:SubsysDecomposition} in Algorithm \ref{Alg:DecouplingProcedure} for selecting a subsystem decomposition, the complexity of the corresponding SDPs (\ref{SDP:MPIDual}) is determined by $\omega$ from Theorem \ref{thm:mainthmformal}.
\end{corollary}

\begin{proof}
    This follows immediately from the convergence results of~\cite{Boss14} and Theorem \ref{thm:GeneralizedDecouplingProcedure}. The complexity statement follows because the largest occurring SDP, i.e. the SDP involving the most variables, is induced by the subsystems containing the most states. Its complexity is determined by the number of states in the subsystem. By Theorem \ref{thm:mainthmformal}, that number is given by $\omega$.
\end{proof}

\begin{remark}
    In Corollary \ref{cor:SDPSparse}, there is a dependence on an underlying decomposition of $X$ according to a family $\mathcal{J}$. We suggest to use $\mathcal{J}$ from Lemma \ref{lem:MinimalFactorization}.
\end{remark}

The integer $\omega$ from Theorem \ref{thm:mainthmformal} is the largest dimension of appearing subsystems. The SDPs obtained by the SOS hierarchy grow combinatorially in the number of variables and the degree bound $k$. For an explicit illustration of the effectiveness of exploiting sparsity let $x_{I_1},\ldots,x_{I_m}$ be a subsystem decomposition for which $X$ decomposes accordingly via $X_1,\ldots,X_m$. Let $n_j := |I_j|$ for $1\leq j \leq m$ and $\nu_j$ be the number of constraints defining the set $X_j$. Let $\nu$ be the total number of constraints defining $X$.
The number of variables in the non-sparse SDP for the full system and the sparse SDPs are of the order
\begin{equation}\label{eq:numbvar}
	(5+ 3\nu)\binom{n + \frac{k}{2}}{\frac{k}{2}} \quad \text{and} \quad \sum\limits_{j = 1}^m (5 + 3\nu_j)\binom{n_j + \frac{k}{2}}{\frac{k}{2}}.
\end{equation}
The reduction is significant if there is strong separation in the dynamics as well as in the constraint set $X$, i.e. $n_j$ is small compared to $n$, and $\nu_j$ is small compared to $\nu$. This is expected because strong separation tells that fewer interactions are needed to describe the system.

\begin{remark}
    We only present the sparse computation of the MPI set based on~(\ref{LP:MPIDual}). The treatment of the ROA and GA is analogous.
\end{remark}

\subsection{Sparse improvement}\label{SubSec:SparseImprovement}
We propose a slightly adapted LP that allows further improvement, leading to tighter outer approximations while maintaining the reduced computational complexity. Again, we restrict to the MPI set but the ROA and GA are treated analogously. For the rest of this section assume that $x_{I_1},\ldots,x_{I_m}$ is a subsystem decomposition and $X$ decomposes according to $I_1,\ldots,I_m$ via sets $X_j \subset \R^{|I_j|}$ for $j = 1,\ldots,m$.

We propose the following sparse version of the LP (\ref{LP:MPIDual}), acting only on functions defined on the variables $x_{I_j}$ for $j = 1,\ldots,m$

\begin{align}\label{LP:MPIDualSparse}
	\begin{tabular}{lcl}
		$s^{*} :=$ & $\inf\limits_{w_j,v_j}$ & $\displaystyle\int\limits_X \sum\limits_{j = 1}^m w_{j}(y) \; d\lambda(y)$\\
		& s.t. & $v_{j},w_j \in \Cc^{1}(\R^{|I_j|}), \; 1\leq j \leq m$\\
                 & & $(w,v)$ is feasible for (\ref{LP:MPIDual}) with\\
                 & & $w(x):= \sum\limits_{j= 1}^m w_j(x_{I_j})\;\;$ and\\
                 & & $v(x):= \sum\limits_{j= 1}^m v_j(x_{I_j})$.
	\end{tabular}
\end{align}

For the corresponding SDP tightening of (\ref{LP:MPIDualSparse}), we choose the SOS multiplier to only depend on the variables $x_{I_l}$.

\begin{remark} For a feasible point $(w_1,v_1,w_2,v_2,\ldots,w_m,v_m)$ of (\ref{LP:MPIDualSparse}) the pairs $(w_j,v_j)$ do not need to be feasible points for the LP (\ref{LP:MPIDual}) for the subsystem $x_{I_j}$. Hence (\ref{LP:MPIDualSparse}) enforces less structure compared to simultaneously solving the LP (\ref{LP:MPIDual}) for each of the subsystems. On the other hand, the LP (\ref{LP:MPIDualSparse}) is a tightening of the LP (\ref{LP:MPIDual}) because it restricts to a sparse class of feasible points. This can potentially hamper convergence of the approximations. By intersecting with the approximations coming from the fully decoupled approach we avoid this undesirable property; this is formally stated in Theorem~\ref{thm:BossFix}.
\end{remark}

\begin{theorem}\label{thm:BossFix}
    For $k \in \N$ let $S^{(k)}$ be the outer approximation of the MPI set $M_+$ from Algorithm \ref{Alg:DecouplingProcedure} using the SDP hierarchy from Section \ref{Subsec:SDP}, and let $(w^{(k)}_j,v^{(k)}_j)$, $j = 1,\ldots,m$ be optimal points for the sparse SDPs for~(\ref{LP:MPIDualSparse}). Set
    \begin{equation}\label{eq:BossFixROA}
        Y^{(k)}:= \{x \in X: \sum\limits_{j = 1}^m w^{(k)}_j(\Pi_{I_j}(x)) \geq 1\}.
    \end{equation}    
    Then $S^{(k)} \cap Y^{(k)}$ is a converging (with respect to $d_\lambda$) outer approximation of $M_+$. As in Corollary~\ref{cor:SDPSparse}, the largest occurring SOS multiplier acts on $\max\limits_{1\leq j \leq m} |I_j|$ many variables.
\end{theorem}

\begin{proof}
     Since the LP (\ref{LP:MPIDualSparse}) is a tightening of (\ref{LP:MPIDual}), by Remark \ref{rem:InterpretLPMPI}, it also holds $Y^{(k)} \supset M_+$ and hence $S^{(k)} \supset S^{(k)} \cap Y^{(k)} \supset M_+$. Convergence follows from the convergence of $S^{(k)}$ stated in Corollary \ref{cor:SDPSparse}. By the enforced sparse structure of the SDPs for the sparse LP (\ref{LP:MPIDualSparse}) the largest SOS multiplier corresponds to the subsystem of the largest dimension; that subsystem has $\omega = \max\limits_{1\leq j \leq m} |I_j|$ variables.
\end{proof}

\section{Numerical examples}

\subsection{Cherry structure}
We call a sparsity graph as in Figure~\ref{fig:vpCherry} a cherry structure. Those structures are the most sparse structures for our framework. They occur for instance in Dubins car and the 6D Acrobatic Quadrotor \cite{Chen}. Here we treat a larger example of a cherry structure. We consider the interconnection of Van der Pol oscillators as in Figure~\ref{fig:vpCherry}.
\begin{figure}[!t]
\begin{center}
\includegraphics[scale=0.5]{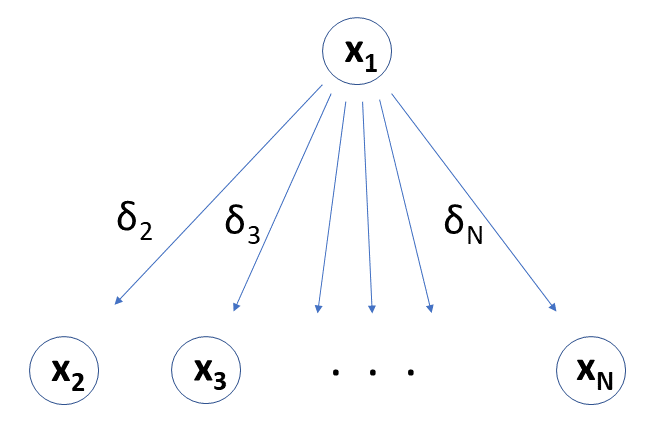}
	\caption{Interconnection of Van-Der Pol oscillators in a cherry structure.}
	\label{fig:vpCherry}
\end{center}
\end{figure}
For the leaf nodes $x^2,\ldots,x^N$, the dynamics is
\begin{align*}
	\dot{x}^i_1 & =  2x^i_2\\
	\dot{x}^i_2 & =  -0.8x^i_1 - 10 [ (x^i_1)^2-0.21]x^i_2 + \delta_i x_1^1.
\end{align*}
For the root note $x^1$, the dynamics is 
\begin{align*}
	\dot{x}^1_1 & =  2x^1_2\\
	\dot{x}^1_2 & =  -0.8x^1_1 - 10 [ (x^1_1)^2-0.21 ]x^1_2.
\end{align*}

We illustrate the decoupling procedure by computing outer approximations of the MPI set of this system with respect to the constraint set $[-1.2,1.2]^{2N}$. We carry out the computation for degree $k = 8$ and $N = 10$, resulting in a total dimension of the state-space equal to 20. The optimal decoupling in this case is into subsystems $(x^1,x^i)$, $i=2,\ldots,N$, each of dimension four. Figure~\ref{fig:VanderPol} shows the sections of the MPI set outer approximations when the value at the root node is fixed at $[0.5,-0.1]$. The computation time was 12 seconds.\footnote{All computations were carried out using YALMIP \cite{Lofberg} and MOSEK running on Matlab and 4.2 GHz Intel Core i7, 32 GB 2400MHz DDR4.} Next we carried out the computation with $k=8$ and $N = 26$, resulting in state-space dimension of 52. Figure \ref{fig:VanderPol52} shows the sections of the MPI set outer approximations when the value at the root node is fixed at $[0.5,-0.1]$. The total computation time was 40.3 seconds. It should be mentioned that these problems in dimensions 20 or 52 are currently intractable without structure exploitation. Here the sparse structure allowed for decoupling in $9$ respectively $25$ problems in 4 variables, which were solved in less than a minute in total.

\begin{figure}[!t]
\begin{center}
\includegraphics[scale=0.24]{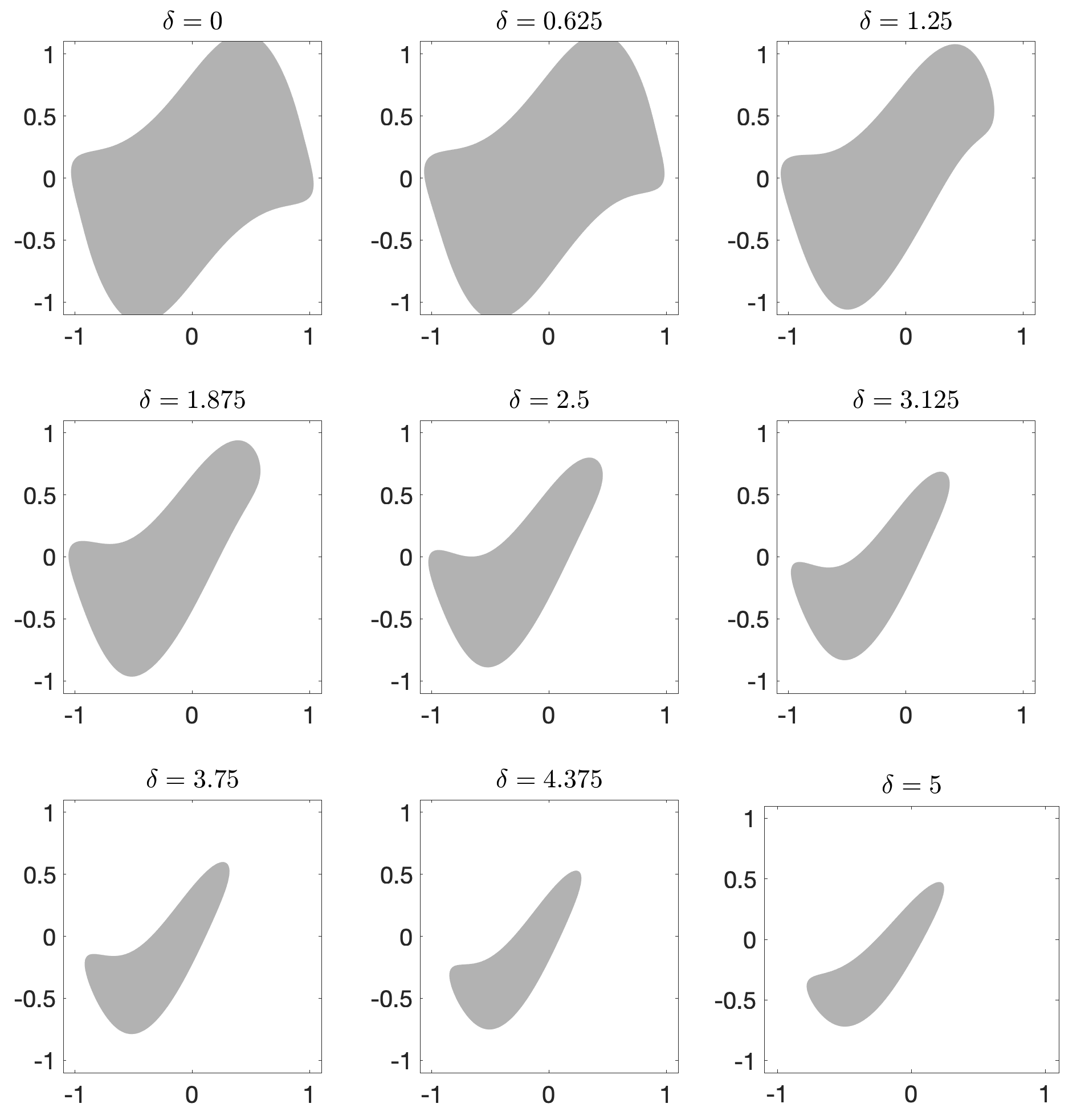}
\caption{\footnotesize{Van der Pol oscillators in a cherry structure: The figure shows the outer approximations of the MPI set for $k = 8$ and $N =10$ for the subsystems given by the cherry--branches.}}
\label{fig:VanderPol}
\end{center}
\end{figure}
\begin{figure}[!t]
\begin{center}
\includegraphics[width=0.9\columnwidth]{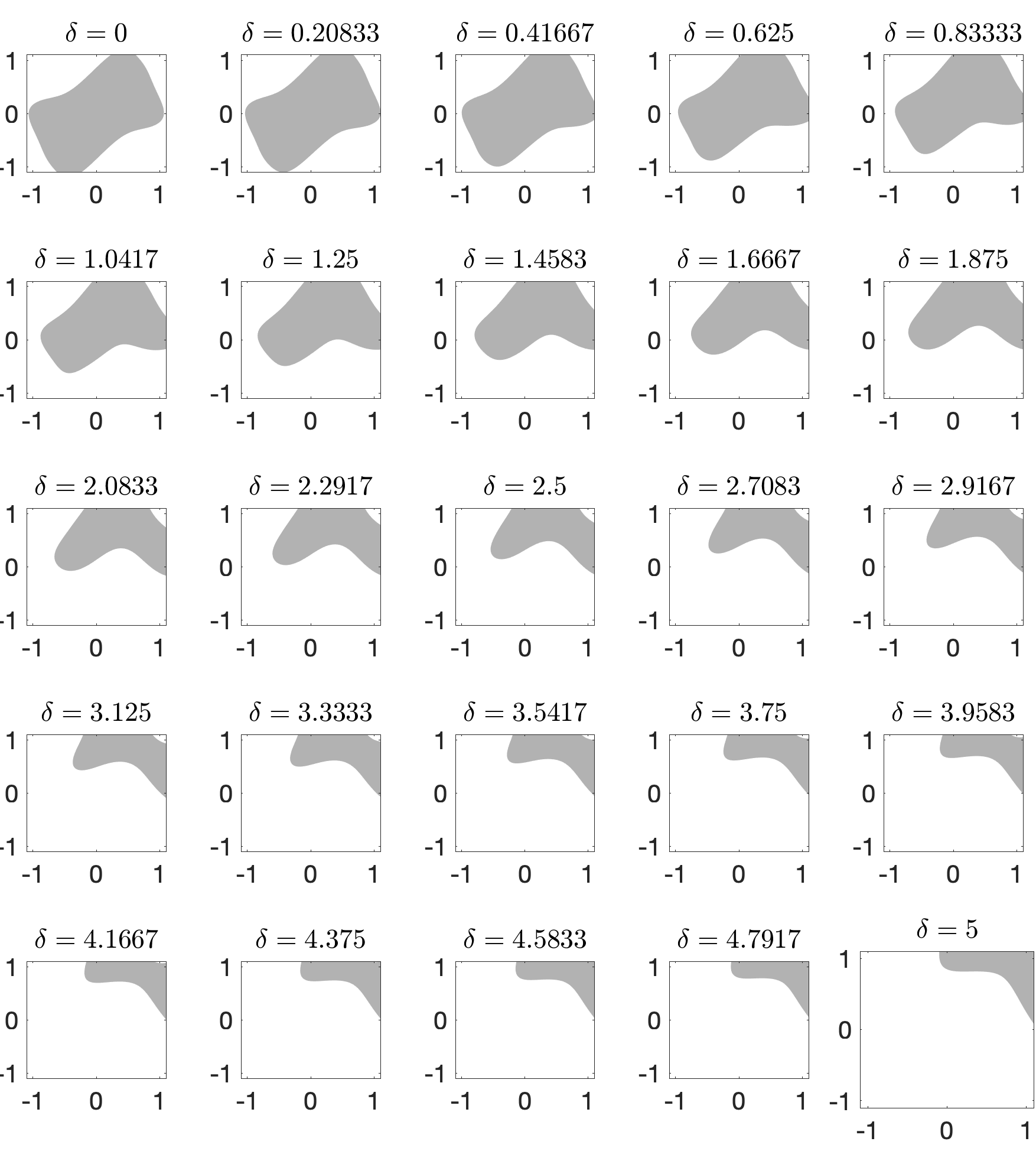}
\caption{\footnotesize{Van der Pol oscillators in a cherry structure: The figure shows the outer approximations of the MPI set for $k = 8$ and $N =26$ for the subsystems given by the cherry--branches}}
\label{fig:VanderPol52}
\end{center}
\end{figure}

\subsection{Tree structure}

We thank Edgar Fuentes for pointing out to us that radial distribution networks provide common examples of tree structures \cite{RadialDistributionNetworks}. Similarly, to the previous example we consider a network of Van der Pol oscillators; now interconnected as in Figure~\ref{fig:vp2}.

The coupling is as in the previous example from the first component of the predecessor state to the second component of the successor state. The coupling intensity $\delta$ is set to 0.1 for each edge. The goal is to compute the MPI set with respect to the constraint set $[-1.2,1.2]^{10}$. The optimal decoupling is now into 3 subsystems given by $(x^1,x^2,x^4)$, $(x^1,x^2,x^5)$, $(x^1,x^3)$; the respective dimensions are $6$, $6$ and $4$. Figure~\ref{fig:vp2RP6} shows six random sections of the ten-dimensional MPI set outer approximation computed by our approach with degree $k = 8$. Even though the overall state-space dimension 10 is less than it was in our previous example, the computation time of 285 seconds is higher since the maximum dimension of the subsystems is higher.

\begin{figure}[!t]
\begin{center}
\includegraphics[width=0.5\columnwidth]{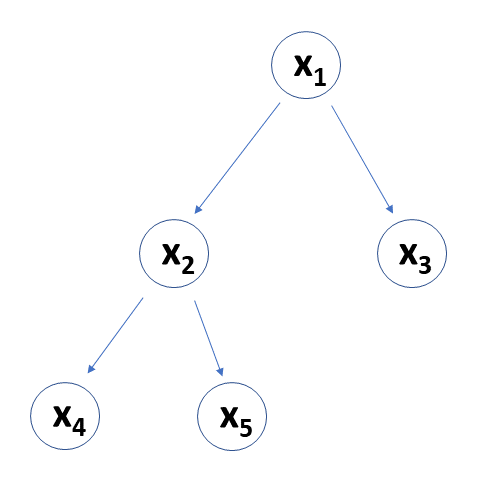}
	\caption{Interconnection of Van-Der Pol oscillators in a tree structure.}
	\label{fig:vp2}
\end{center}
\end{figure}

\begin{figure}[!t]
\begin{center}
\includegraphics[width=0.6\columnwidth]{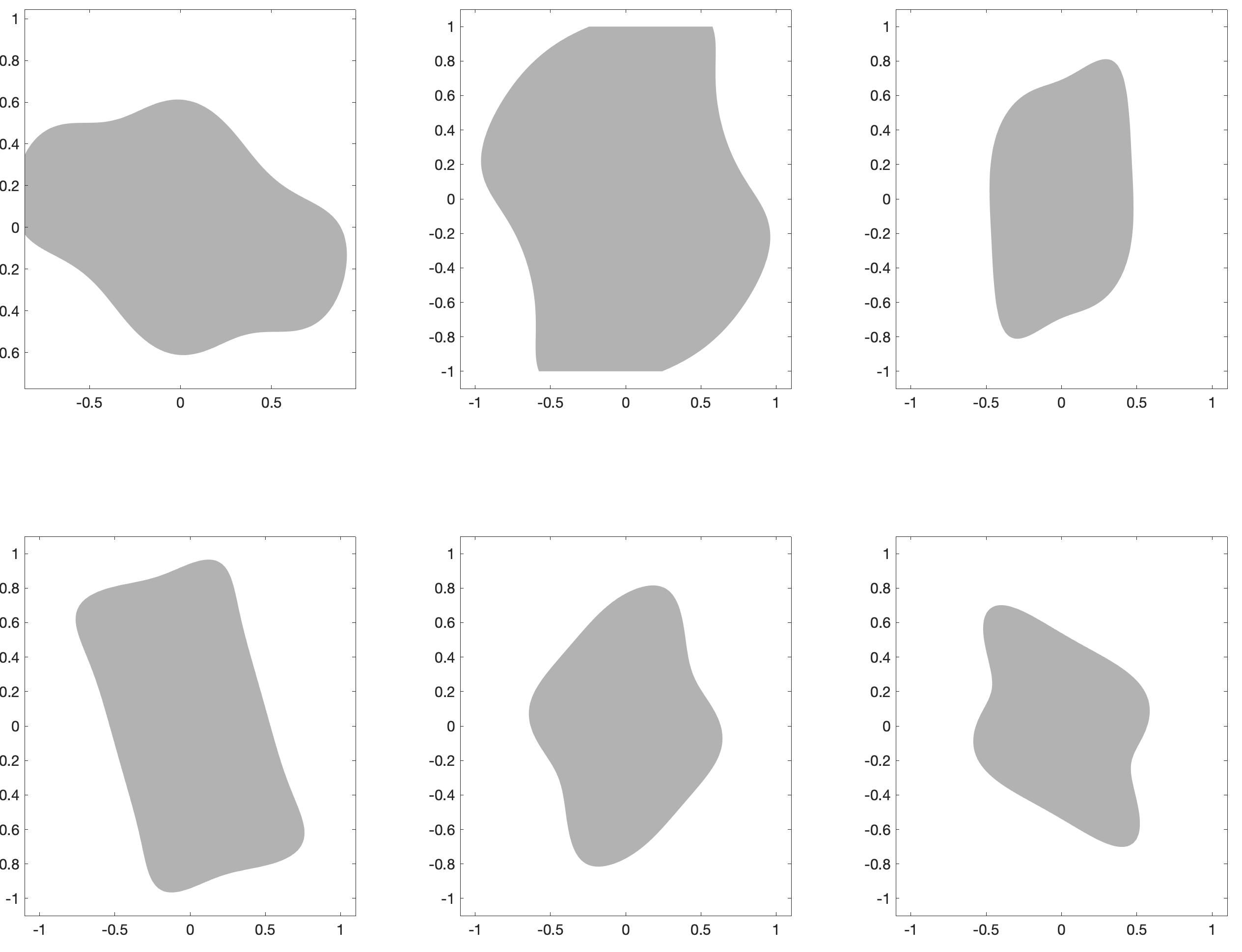}
	\caption{\footnotesize Van der Pol oscillators in a tree structure: Random projections of the outer approximation to the ten-dimensional MPI set.}
	\label{fig:vp2RP6}
\end{center}
\end{figure}

\begin{figure}[!t]
\begin{center}\label{Fig:Cascade}
\includegraphics[scale=0.3]{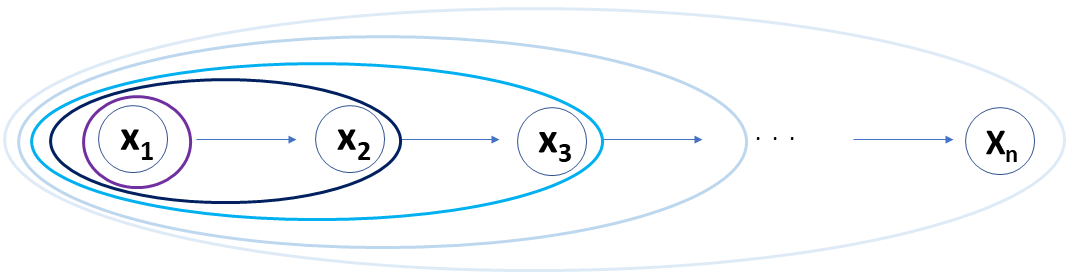}
\caption{\footnotesize{Cascade without branching. Circles around states indicate subsystems, i.e. all the subsystems are induced by index sets of the form $I = \{1,\ldots,k\}$ for $1\leq k \leq n$.}}
\label{FigCycleSystem}
\end{center}
\end{figure}

\begin{remark}\label{rem:CascadeDecomposition}
   We like to mention the following observation for systems where our approach does not obtain an overall reduction in the dimension. As an example, let us consider a cascade system as in Figure~\ref{Fig:Cascade}, i.e. $\dot{x}_1 = f(x_1)$ and $\dot{x}_k = f_k(x_1,\ldots,x_k)$ for $k = 2,\ldots,n$. A subsystem decomposition is given by $I_1 := \{1\}, I_2 = \{1,2\},\ldots,I_n := \{1,\ldots,n\}$. Even though $I_n$ corresponds to the full system we benefit from lower dimensional subsystems $x_{I_k}$ via Theorem \ref{thm:DecouplingRoAMPIAttractorGeneral}. Additionally, in the SOS approach, for smaller $1\leq k\leq n$ higher degree bounds can be used. This induces higher accuracy in the approximation of the set of interest for the subsystems, and, by Theorem \ref{thm:GeneralizedDecouplingProcedure}, improves the overall approximation obtained from Algorithm \ref{Alg:DecouplingProcedure}. More generally, a subsystem decomposition allows taking specific properties and structures of the subsystems into account and therefore can simplify the analysis of the whole system.
\end{remark}

\section{Conclusion}
\color{black}
We present a decomposition of nonlinear dynamical systems into what we call subsystems. We show that several important sets for dynamical systems decompose with respect to their correspondences for the subsystems. Our work was motivated by~\cite{Chen} and extended from the region of attraction to the maximum positively invariant set as well as to the global attractor. For these sets, we utilize the decomposition to propose computational methods that operate solely on the subsystems and can be integrated with existing methods for computing or approximating those sets. As an illustrative example, we pair the decomposition with an existing convex optimization approach to approximate the region of attraction, the maximum positively invariant set, and the global attractor. We show that this induces a converging sequence of outer approximations that exploits the sparse structure of the dynamical system and reduces computational cost significantly when the system is sparse.

We believe that decomposing the dynamical system into subsystems as presented here can be beneficial for other objectives such as invariant manifolds or the constructions of Lyapunov functions. It may also be of interest to exploit sparsity for extreme value computation, building on~\cite{Goluskin}. Another direction of future work is the inclusion of control, i.e., the computation of the region of attraction with control, the maximum controlled invariant set, and optimal control. Utilizing this approach in a data-driven setting, building on~\cite{korda2020}, is another possible generalization.

Sparsity in the dependence of the dynamics of the states is not the only structure of $f$ that can be exploited. For example, in \cite{JiePaper}, the algebraic structure of polynomial dynamics $f$ has been investigated from the perspective of term-sparsity. In addition, more general sparse structures can and should be investigated as we have seen that our approach treats straight paths or cycles as subsystems -- in the same way as if all the corresponding nodes were fully connected. Work in this direction was done in~\cite{Matteo}.

Additional reduction techniques, such as symmetry, can be combined with our approach. Decoupling into subsystems maintains symmetry structures (for the subsystems), so merging our approach with for example the symmetry argument in \cite{FantuzziGoluskin} can be done by first decoupling into subsystems and then exploiting symmetries of the subsystems (independently).

Our notion of subsystems is coordinate-dependent. A coordinate-free approach is desirable but challenging from a computational perspective. We aim to investigate a coordinate-free formulation in future work.

\appendix
\section{Appendix:}

\begin{proof}\textit{ Of Lemma \ref{lem:MinimalFactorization}.}
	We look at the set
	\begin{eqnarray*}
	    T:= \{J\subset \{1,\ldots,n\} & : & J \text{ and } \{1,\ldots,n\}\setminus  J \text{ induce } \\
	    & & \text{a factorization of } X\}.
	\end{eqnarray*}
	The set $T$ is the collection of all partitions consisting of only two sets, such that they induce a factorization of $X$. We will see that $T$ contains minimal elements (with respect to inclusion); these will give rise to the desired factorization of $X$. We start with the following properties of $T$.
	\begin{enumerate}
		\item $T$ is non-empty.\\
		$J = \{1,\ldots,n\}$ is contained in $X$ because it induces the trivial factorization $X$ of factoring into itself.
		\item $T$ is closed with respect to taking the complement in $\{1,\ldots,n\}$.\\
		Let $J \in T$ then $J^c := \{1,\ldots,n\} \setminus J \in T$ because the pair $J^c,J$ is a partition that induces the same factorization as the pair $J,J^c$.
		\item $T$ is closed with respect to intersections.\\
		Let $J_1,J_2 \in T$ with corresponding sets $X_1 := \Pi_{J_1}(X)$, $X_2 :=\Pi_{J_1^c}(X)$ and $Y_1 := \Pi_{J_2}(X) , Y_2 :=\Pi_{J_2^c}(X)$. Let $J:= J_1 \cap J_2$ and $I:= \{1,\ldots,n\} \setminus J$. We claim $J,I$ induces a factorization. Therefore let $Z_1:= \Pi_J(X)$ and $Z_2 := \Pi_I(X)$. We need to show that we have
		\begin{equation}\label{eq:FactorizationClaim}
			X = X':=  \{x \in \R^n : \Pi_J(x) \in Z_1, \Pi_I(x) \in Z_2\}.
		\end{equation}
		For any $x \in X$ we have $x \in X'$ by definition of $Z_1$ and $Z_2$. To shoe $X' \subset X$, let $x' \in X'$. By definition of $Z_1$ there exists $x_1 \in X$ with $\Pi_J(x_1) = \Pi_J(x')$. The idea is to modify $x_1$ to $x'$ only by operations that maintain membership to $X$. We will do so using that $J_1$ and $J_2$ induce factorizations. From $J_1 \in T$ it follows $\Pi_{J_1}(x_1) \in X_1$. Since $J_1,J_1^c$ induces a factorization the element $x_2 \in \R^n$ with $\Pi_{J_1}(x_2) = \Pi_{J_1}(x_1)$ and $\Pi_{J_1^c}(x_2) = \Pi_{J_1^c}(\Pi_I(x'))$ belongs to $X$. If we repeat this process with $J_1$ replaced by $J_2$ we find an element $x_3 \in X$ such that $\Pi_J(x_3) = \Pi_{J_1 \cap J_2} (x_3) = \Pi_{J_1 \cap J_2}(x')$ and $\Pi_I(x_3) = \Pi_{J_1^c \cup J_2^c} (x_3) = \Pi_{J_1^c \cup J_2^c}(x')$, i.e. $x' = x_3 \in X$.
		\item $T$ is closed with respect to taking unions.\\
		Let $J_1,J_2 \in T$. Then $J_1 \cup J_2 = (J_1^c \cap J_2^c)^c \in T$.
	\end{enumerate}
	It follows that $T$ is a (finite) topology and hence there exists a minimal basis of $T$ (consisting of the smallest neighbourhoods of each point), i.e. for each $i \in \{1,\ldots,n\}$ define $U_i := \bigcap\limits_{J \in T: i \in J} J \in T$. Those $U_i$ are minimal elements in $T$ containing $i$, and hence their unions covers $\{1,\ldots,n\}$. Further for $i \neq k$ the sets $U_i$ and $U_k$ are either identical or disjoint, otherwise intersecting them would create smaller non-empty elements in $T$. Because some of the sets $U_i$ might coincide and we want to work with partitions we have to remove redundant sets among $U_i$. Let $J_1,\ldots,J_N$ be the partition consisting of sets $U_i$, i.e. $\{J_1,\ldots,J_N\} = \{U_1,\ldots,U_n\}$. Note that $J_i \cap J_l = \emptyset$ for all $i \neq l$ and $\bigcup\limits_{l = 1}^N J_l = \bigcup\limits_{i = 1}^n U_i = \{1,\ldots,n\}$, i.e. $J_1,\ldots,J_N$ is a partition of $\{1,\ldots,n\}$. We claim that this defines the finest partition that factorizes $X$. We show first that $J_1,\ldots,J_N$ induces a factorization. For each $1\leq l \leq N$ there exist sets $X_l$ (and $X_l'$) such that
	\begin{equation}
	    X = \{x \in \R^n: \Pi_{J_l}(x) \in X_l, \Pi_{J_l^c}(x) \in X_l'\}.
	\end{equation} 
	We claim $X = \{x \in \R^n: \Pi_{J_l}(x) \in X_l \text{ for } l = 1,\ldots,N\}$. It suffices to show that $\{x \in \R^n: \Pi_{J_l}(x) \in X_l \text{ for } l = 1,\ldots,N\} \subset X$. Therefore let $x \in \R^n$ such that $\Pi_{J_l}(x) \in X_l$.
	Because $J_2 \in T$ it follows from $\Pi_{J_2}(x) \in X_2$  that there exists a $x^{(2)} \in X$ with $\Pi_{J_2}(x^{(2)}) = \Pi_{J_2}(x)$. Hence it follows $\Pi_{J_1^c}(x^{(2)}) \in X'_1$. In particular the element
	\begin{equation}
	    \tilde{x}^{(2)} = (\tilde{x}^{(2)}_i)_{i = 1,\ldots,n} \text{ with } \tilde{x}^{(2)}_i =  \begin{cases} x_i, & i \in J_1\\
	                  x^{(2)}_i, & i \in J_1^c 
	    \end{cases}
	\end{equation}
	belongs to $X$ and satisfies $\tilde{x}^{(2)}_i = x_i$ for $i \in J_1 \cup J_2$. Now we can continue this process for the new partition $(J_1 \cup J_2),J_3,\ldots,J_N$ and find an element $\tilde{x}^{(3)} \in X$ with $\tilde{x}^{(2)}_i = x_i$ for $i \in J_1 \cup J_2 \cup J_3$. Continuing until we have reached $J_N$ we find that finally $x = \tilde{x}^{(N)} \in X$. It remains to show that $J_1,\ldots,J_N$ is minimal. Therefore, let $1\leq k \leq M$. Then $I_k,I_k^c$ induces a factorization because $I_1,\ldots,I_M$ already induces a factorization, because $I_k^c = \bigcup\limits_{r \neq k} I_r$. That means $I_k \in T$ and since the $U_i$ build a basis we have $I_k = \bigcup\limits_{i \in I_k} U_i = \bigcup\limits_{J_l \subset I_k} J_l$, what remained to be shown.
\end{proof}

\end{document}